\documentclass[12pt]{amsart}
\usepackage{SSdefn}
\usepackage{eucal}
\usepackage[shortlabels]{enumitem}
\usepackage{tikz}

\newcommand{\psurj}{\rightarrowtail}

\newcommand{\Langle}{\langle\!\langle}
\newcommand{\Rangle}{\rangle\!\rangle}

\newcommand{\fin}{\mathrm{fin}}

\renewcommand{\top}{\mathrm{top}}

\title{Symmetric subvarieties of infinite affine space}

\author{Rohit Nagpal}
\address{Department of Mathematics, University of Michigan, Ann Arbor, MI}
\email{\href{mailto:rohitna@umich.edu}{rohitna@umich.edu}}
\urladdr{\url{http://www-personal.umich.edu/~rohitna/}}

\author{Andrew Snowden}
\address{Department of Mathematics, University of Michigan, Ann Arbor, MI}
\email{\href{mailto:asnowden@umich.edu}{asnowden@umich.edu}}
\urladdr{\url{http://www-personal.umich.edu/~asnowden/}}

\thanks{RN was partially supported by NSF DMS-1638352. AS was supported by NSF DMS-1453893.}

\date{\today}

\begin{document}

\begin{abstract}
We classify the subvarieties of infinite dimensional affine space that are stable under the infinite symmetric group. We determine the defining equations and point sets of these varieties as well as the containments between them.
\end{abstract}

\maketitle

\tableofcontents

\section{Introduction}

Cohen \cite{cohen,cohen2} proved that ideals in the infinite variable polynomial ring $R=\bC[\xi_1,\xi_2,\ldots]$ that are stable under the infinite symmetric group $\fS$ satisfy the ascending chain condition; in other words, $R$ is $\fS$-noetherian. This suggests that $\fS$-equivariant commutative algebra over $R$ should be well-behaved. We have undertaken a detailed study of this theory, and found that this is indeed the case. In this paper, we classify the radical $\fS$-ideals of $R$. In subsequent papers \cite{sideals,smod}, we will build on the results of this paper to determine the structure of arbitrary $\fS$-ideals of $R$, and establish results such as primary decomposition. Our work is motivated by the wide variety of applications Cohen's theorem has found in the last decade (e.g., \cite{AschenbrennerHillar, DraismaEggermont, DraismaKuttler, GunturkunNagel, HillarSullivant, KLS, LNNR, NagelRomer}); see \cite{draisma-notes} for a good introduction.

\subsection{The classification theorem} \label{ss:descr}

Let $\fX=\Spec(R)=\bA^{\infty}$ be infinite dimensional affine space over the complex numbers (we allow arbitrary noetherian coefficient rings in the body of the paper). The goal of this paper is to classify the $\fS$-stable Zariski closed subsets of $\fX$ and understand their structure.

We say that a $\bC$-point $x=(x_i)_{i \ge 1}$ of $\fX$ is \textbf{finitary} if the coordinates of $x$ assume only finitely many values, that is, $\{x_i \mid i \ge 1 \}$ is a finite subset of $\bC$. This notion is closely related to discriminants. Indeed, let $\Delta_r=\prod_{1 \le i<j \le r} (\xi_i-\xi_j)$ be the $r$th discriminant. Then the point $x$ is finitary if and only if there is some $r$ such that $\Delta_r(\sigma x)=0$ for all $\sigma \in \fS$. It is not difficult to show that any non-zero $\fS$-stable ideal of $R$ contains some discriminant (Proposition~\ref{prop:discriminant}). Thus any proper $\fS$-stable closed subset of $\fX$ is contained in the finitary locus $\fX_{\fin}$. For this reason, finitary points play a prominent role in this paper.

Suppose that $x$ is a finitary point. Let $a_1, \ldots, a_r$ be the distinct values its coordinates assume, and let $\lambda_i$ be the number of times $a_i$ appears. Relabeling if necessary, we assume that $\lambda_1 \ge \cdots \ge \lambda_r$; note that $\lambda_1$ is necessarily $\infty$. We call $\lambda=(\lambda_1, \ldots, \lambda_r)$ the {\bf type} of $x$, and generally refer to a tuple of this sort as an {\bf $\infty$-parition}. For example, the point
\begin{displaymath}
(3,3,3,5,5,6,7,6,7,\ldots,6,7,\ldots)
\end{displaymath}
has type $(\infty,\infty,3,2)$, since 6 and 7 occur infinitely, 3 occurs thrice, and 5 occurs twice. Let $\fX_{[\lambda]}$ be the set of all finitary points of type $\lambda$. We have $\fX_{\fin}=\bigsqcup \fX_{[\lambda]}$, where the union is over all $\infty$-partitions $\lambda$, and so we can regard the $\fX_{[\lambda]}$ as defining a stratification of $\fX_{\fin}$.

Before stating our theorem, we need to introduce one more concept. We say that a topological space $X$ with an action of a group $G$ is {\bf $G$-irreducible} if it cannot be expressed as a union of two $G$-stable proper closed subsets. It is an easy consequence of Cohen's theorem that any $\fS$-stable closed subvariety of $\fX$ is a finite union of $\fS$-irreducible closed subvarieties. It therefore suffices to study the $\fS$-irreducible subvarieties of $\fX$. We note that if $G$ is a finite group acting on a variety $X$ then any $G$-irreducible subvariety of $X$ is simply the union of the $G$-translates of an irreducible subvariety.

We can now state our classification theorem:

\begin{theorem}[Theorem~\ref{thm:classification-fZ}] \label{mainthm2}
The following two sets are in natural bijection:
\begin{enumerate}
\item The set of $\fS$-irreducible proper closed subvarieties $\fZ$ of $\fX$.
\item The set of pairs $(\lambda, Z)$, where $\lambda=(\lambda_1, \ldots, \lambda_r)$ is an $\infty$-partition and $Z$ is an $\Aut(\lambda)$-irreducible closed subvariety of $\bA^{[r]}$. Here $\Aut(\lambda)$ is the subgroup of $\fS_r$ stabilizing $\lambda$, and $\bA^{[r]}$ is the open subvariety of $\bA^r$ consisting of points with distinct coordinates.
\end{enumerate}
\end{theorem}

It is not difficult to describe the bijection in Theorem~\ref{mainthm2}. To do this, it will be useful to introduce a refinement of the type stratification of $\fX_{\fin}$. Let $\cU=\{\cU_1,\ldots,\cU_r\}$ be a partition of the index set $[\infty]=\{1,2,\ldots\}$ into finitely many disjoint and non-empty pieces. Define $\fX_{\cU}$ (resp.\ $\fX_{[\cU]}$) be the subset of $\fX$ consisting of points $x$ such that $x_i=x_j$ if $i$ and $j$ belong to the same part of $\cU$ (resp.\ $x_i=x_j$ if and only if $i$ and $j$ belong to the same part). The space $\fX_{\cU}$ is isomorphic to $\bA^r$, and under this isomorphism $\fX_{[\cU]}$ corresponds to $\bA^{[r]}$. Let $\lambda_i=\# \cU_i$, relabeling if necessary so that $\lambda=(\lambda_1, \ldots, \lambda_r)$ is an $\infty$-partition; we call this the $\infty$-partition associated to $\cU$, and say $\cU$ is of type $\lambda$. Then any point of $\fX_{[\cU]}$ has type $\lambda$, and $\fX_{[\lambda]}$ is the union of the $\fS$-translates of $\fX_{[\cU]}$.

We now describe the bijection in Theorem~\ref{mainthm2}. Suppose that $\fZ$ is an $\fS$-irreducible closed subvariety of $\fX$. Let $\Lambda_{\fZ}$ be the set of all $\infty$-partitions $\lambda$ such that $\fZ$ has a point of type $\lambda$. We show that $\Lambda_{\fZ}$ contains a unique maximal element with respect to a partial order $\preceq$ on $\infty$-partitions. Let $\lambda=(\lambda_1, \ldots, \lambda_r)$ be this maximal element, and pick a partition $\cU$ of $[\infty]$ of type $\lambda$. Let $Z=\fZ \cap \fX_{[\cU]}$, where we identify $\fX_{[\cU]}$ with $\bA^{[r]}$. Then $\fZ$ corresponds to $(\lambda, Z)$. In the reverse direction, suppose that $\lambda=(\lambda_1, \ldots, \lambda_r)$ and $Z$ are given. Again, pick $\cU$ of type $\lambda$. Then the corresponding $\fS$-irreducible closed subvariety of $\fX$ is the Zariski closure of the $\fS$-orbit of $Z$, where we regard $Z$ as a subset of $\fX$ via the identification $\bA^{[r]}=\fX_{[\cU]}$.

\subsection{C3 varieties}

While Theorem~\ref{mainthm2} does give a complete classification of $\fS$-subvarieties of $\fX$, it is not powerful enough to answer some finer questions, such as:
\begin{itemize}
\item How does one describe the entire point set of $\fZ$ from $(\lambda,Z)$?
\item How does one determine the equations of $\fZ$ from those of $Z$?
\item How does one determine a containment $\fZ \subset \fZ'$ in terms of the data $(\lambda,Z)$ and $(\lambda',Z')$?
\end{itemize}
In fact, Theorem~\ref{mainthm2} is a corollory of a stronger result, that we now describe, which allows us to address these questions (and more like them).

For an $\infty$-partition $\lambda=(\lambda_1, \ldots, \lambda_r)$, put $\cX_{\lambda}=\bA^r$. We regard the system $\cX=\{\cX_{\lambda}\}_{\lambda}$ as a single object. Given an $\fS$-stable closed subset $\fZ$ of $\fX$, we define a subsystem $\cZ=\Phi(\fZ)$ of $\cX$ as follows. For an $\infty$-partition $\lambda$, let $\cU$ be a partition of $[\infty]$ of type $\lambda$, and put $\cZ_{\lambda}=\fZ \cap \fX_{\cU}$, where, as usual, we idenify $\fX_{\cU}$ with $\bA^r=\cX_{\lambda}$. This is well-defined (i.e., independent of the choice of $\cU$) since $\fZ$ is $\fS$-stable. Note that $\cZ_{\lambda}$ is Zariski closed in $\cX_{\lambda}$, and so we can regard $\cZ$ as a system of finite dimensional varieties.

Since $\Phi(\fZ)$ records all the finitary points of $\fZ$, it follows that it completely determines $\fZ$. In other words, the construction $\Phi$ is injective. Much of the work in this paper goes into determining the image of $\Phi$, or, in other words, determining the precise conditions that the system $\cZ$ of varieties must satisfy.

We now give an example of one of the simplest conditions that $\cZ$ must satisfy. Suppose that $x=(b,a,a,\ldots)$ is a $\bC$-point of $\fZ$. Let $x_n=(a,\ldots,a,b,a,a,\ldots)$, where all coordinates except the $n$th are equal to $a$. Then $x_n$ belongs to the $\fS$-orbit of $x$, and therefore belongs to $\fZ$. The sequence $\{x_n\}_{n \ge 1}$ converges coordinate-wise to $y=(a,a,\ldots)$, and so $y$ also belongs to $\fZ$. (This is a simple compactness argument, see \S \ref{ss:pi}.) We can rephrase the above observation in terms of $\cZ$ as follows. The point $x$ shows that $\cZ_{(\infty,1)} \subset \bA^2$ contains the point $(a,b)$, while the point $y$ shows that $\cZ_{(\infty)} \subset \bA^1$ contains the point $a$. We thus see that if $p_1 \colon \bA^2 \to \bA^1$ is the projection map onto the first coordinate then $p_1(\cZ_{(\infty,1)}) \subset \cZ_{(\infty)}$. This is a non-trivial condition on $\cZ$.

By generalizing the above argument, we are led to a number of natural conditions that $\cZ$ must satisfy. We define a ``C3 subvariety'' of $\cX$ to be a subsystem satisfying these conditions. The following is the main theorem of this paper:

\begin{theorem}[Theorem~\ref{thm:C3-corr}] \label{mainthm1}
The construction $\Phi$ defines a bijection
\begin{displaymath}
\{ \text{$\fS$-stable closed subsets of $\fX$} \} \to \{ \text{C3 subvarieties of $\cX$} \}
\end{displaymath}
\end{theorem}

The inverse bijection $\Psi$ is described excplitily. We use this theorem to answer the questions posed above. Additionally, we apply the theorem in \S \ref{s:support}, where we prove a result on the support of $\fS$-equivariant $R$-modules. This result will play an important role in the follow-up paper \cite{sideals}.

\subsection{Some examples}

We give some simple examples to illustrate Theorem~\ref{mainthm2}.

\begin{example}
Let $\lambda=(\infty,\infty)$ and let $Z \subset \cX_{[\lambda]}$ be the 0-dimensional closed subvariety $\{(0,1),(1,0)\}$. Note that $\Aut(\lambda)=\fS_2$ acts on $\cX_{[\lambda]} \subset \bA^2$ by permuting the two coordinates. We thus see that $Z$ is stable under $\Aut(\lambda)$ and is $\Aut(\lambda)$-irreducible. Let $\fZ$ be the $\fS$-irreducible closed subvariety of $\fX$ corresponding to $(\lambda, Z)$. Then $\fZ$ consists of those points $(x_i)_{i \ge 1}$ of $\fX$ such that $x_i \in \{0,1\}$ for all $i$. See \S \ref{ss:example-C} for details.
\end{example}

\begin{example} \label{example2}
Let $\lambda=(\infty,n)$, with $n$ finite, and let $Z=\{(0,1)\}$. In this case, $\Aut(\lambda)$ is trivial, and so $Z$ is clearly $\Aut(\lambda)$-irreducible. Let $\fZ$ be the $\fS$-irreducible closed subvariety of $\fX$ corresponding to $(\lambda, Z)$. Then $\fZ$ consists of those points $(x_i)_{i \ge 1}$ of $\fX$ such that $x_i \in \{0,1\}$ for all $i$, and $\# \{i \mid x_i=1\} \le n$.
\end{example}

\begin{example}
Once again, take $\lambda=(\infty,n)$. Let $Z$ be an irreducible curve in $\cX_{[\lambda]}$, such as $y^2=x^3+1$. Then $Z$ is $\Aut(\lambda)$-irreducible, as in the previous example. Let $\fZ$ correspond to $(\lambda, Z)$. Then $\fZ$ consists of those points $(x_i)_{i \ge 1}$ satisfying the following condition: there is a point $(a,b) \in \ol{Z}$ such that $x_i \in \{a,b\}$ for all $i$, and $\# \{ i \mid x_i \ne a \} \le n$. Here $\ol{Z}$ denotes the Zariski closure of $Z$ in $\cX_{\lambda}=\bA^2$.
\end{example}

\subsection{Further directions}

In this paper, we classify the $\fS$-stable subvarieties of $\bA^{\infty}$, or, equivalently, the $\fS$-stable radical ideals of $R$. We are currently pursuing two natural generalizations of this result. First, in forthcoming papers \cite{sideals,smod}, we determine (in a certain sense) the structure of arbitrary $\fS$-stable ideals of $R$, and establish results about equivariant modules as well. And second, in ongoing joint work with Vignesh Jagathese, we are classifying the $\fS$-stable subvarieties of $X^{\infty}$ for an arbitrary variety $X$ (the present paper treating the case $X=\bA^1$).

\subsection{Notation}

We list the most important notation used throughout the paper:

\begin{description}[align=right,labelwidth=2cm,leftmargin=!]
\item[ ${[n]}$ ] the set $\{1, \ldots, n\}$ (we allow $n=\infty$)
\item[ $\fS$ ] the (small) symmetric group on $[\infty]$
\item[ $A$ ]  the coefficient ring (often noetherian)
\item[ $W$ ] the spectrum of $A$
\item[ $R$ ] the polynomial ring over $A$ in variables $\{\xi_i\}_{i\ge 1}$ 
\item[ $\fX$ ] the spectrum of $R$
\item[ $\fZ$ ] a subset of $\fX$
\item[ $\cX$ ] the system $\{\cX_{\lambda}\}_{\lambda}$ indexed by $\infty$-compositions $\lambda$
\item[ $\cZ$ ] a subset (really subsystem) of $\cX$
\end{description}

\section{The space $\fX$} \label{s:fS}

\subsection{Definitions}

Let $A$ be a commutative ring and let $W=\Spec(A)$ be its spectrum. Let $R=R_A=A[\xi_i]_{i \ge 1}$ be the infinite variable polynomial ring over $A$ in the variables $\xi_i$. Let $\fX=\fX_W=\Spec(R)$. We regard $\fX$ as a topological space (under the Zariski topology), and include its scheme-theoretic (non-closed) points.

We typically work with points in $\fX$ by using $K$-points, for variable fields $K$. Every point of $\fX$ comes from a $K$-point of $\fX$ for some field $K$, and a $K$-point $x$ and a $K'$-point $x'$ define the same point of $\fX$ if and only if there is a field $K''$ and embeddings $K \to K''$ and $K' \to K''$ such that $x$ and $x'$ define the same $K''$-points of $\fX$. We can therefore define a condition on points of $\fX$ by defining a condition on $K$-points that is invariant under field extension. For a subset $\fZ$ of $\fX$, we let $\fZ(K)$ be the subset of $\fX(K)$ consisting of $K$-points whose image lies in $\fZ$. We note that $X(K)=W(K) \times K^{\infty}$. For readers not familiar with scheme theory, one can assume that $A$ is an algebraically closed field and work only with closed points without losing much.

Let $\fS=\bigcup_{n \ge 1} \fS_n$ be the ``small'' infinite symmetric group and let $\fS^{\rm big}=\Aut([\infty])$ be the ``big'' infinite symmetric group; of course, $\fS$ is a subgroup of $\fS^{\rm big}$. The group $\fS^{\rm big}$ (and therefore $\fS$ as well) acts on $R$ by permuting the variables, and thus on $\fX$ as well. We prefer to work with $\fS$, since it is more algebraic in nature; however, the difference between these two groups is largely irrelevant for algebraic questions, as we see below (e.g., Proposition~\ref{prop:Pi-closure-is-big-stable}).

\subsection{The noetherian property} \label{ss:noeth}

By an \textbf{$\fS$-ideal} of $R$, we mean an ideal of $R$ that is stable under the group $\fS$. Given a subset $S$ of $R$, we let $\Langle S \Rangle$ be the ideal generated by the $\fS$-orbit of $S$; this is the minimal $\fS$-ideal containing $S$. By an \textbf{$\fS$-subset} of $\fX$, we mean a subset of $\fX$ that is stable under $\fS$. The following result is due to Cohen \cite{cohen,cohen2}:

\begin{theorem}
Suppose that the coefficient ring $A$ is noetherian. Then:
\begin{enumerate}
\item The ascending chain condition holds for $\fS$-ideals of $R$.
\item Any $\fS$-ideal of $R$ is generated by finitely many $\fS$-orbits of elements.
\item The descending chain condition holds for Zariski closed $\fS$-subsets of $\fX$.
\end{enumerate}
\end{theorem}

\begin{corollary} \label{cor:fX-components}
Let $\fZ$ be a Zariski closed $\fS$-subset of $\fX$. Then $\fZ=\fZ_1 \cup \cdots \cup \fZ_r$ where each $\fZ_i$ is an $\fS$-irreducible closed subset of $\fX$.
\end{corollary}

\begin{proof}
The usual argument (from the non-equivariant case) applies. (In fact, by passing to the quotient space $\fX/\fS$ one reduces to that case.)
\end{proof}

\subsection{The $\Pi$-topology} \label{ss:pi}

The space $\fX$ carries the usual Zariski topology. However, it also has a second topology, what we call the {\bf $\Pi$-topology}, that will be useful in this paper. The definition is as follows: a subset $\fZ$ of $\fX$ is $\Pi$-closed if $\fZ(K)$ is closed in $\fX(K) = W(K) \times K^{\infty}$ for all fields $K$; here $W(K)$ and $K$ are given the discrete topology and $\fX(K)$ is given the product topology. Concretely, a sequence (or net) $\{x_i\}$ in $\fX(K)$ converges to a point $x$ if each coordinate of $x_i$ (including the $W$ coordinate) is equal to the corresponding coordinate of $x$ for $i \gg 0$.

\begin{proposition}
\label{prop:Zariski-is-coarser}
Every Zariski closed subset of $\fX$ is also $\Pi$-closed.
\end{proposition}

\begin{proof}
Let $\fZ$ be a Zariski-closed subset of $\fX$ and let $\{x_i\}$ be a net in $\fZ(K)$ converging to a point $x \in \fX(K)$ in the $\Pi$-topology. Suppose $f \in R$ vanishes on $\fZ$, and uses the variables $\xi_1, \ldots, \xi_n$. By definition, there is some $i_0$ such that the first $n$ coordinates of $x_i$ agree with those of $x$ for all $i \ge i_0$. We thus see that $f(x)=f(x_i)=0$ for any $i \ge i_0$. Since all functions that vanish on $\fZ$ vanish on $x$, it follows that $x \in \fZ$. Thus $\fZ$ is $\Pi$-closed.
\end{proof}

\begin{proposition}
\label{prop:Pi-closure-is-big-stable}
Let $\fZ$ be a $\Pi$-closed $\fS$-subset of $\fX$. Then $\fZ$ is $\fS^{\rm big}$-stable.
\end{proposition}

\begin{proof}
Let $x \in \fZ(K)$, and let $\sigma \in \fS^{\rm big}$. For each $n \ge 1$, pick $\tau_n \in \fS$ such that $\tau_n$ agrees with $\sigma$ on $[n]$. Then $\tau_n x \in \fZ$ and $\tau_n x$ converges to $\sigma x$ in the $\Pi$-topology. Since $\fZ$ is $\Pi$-closed, we see that $\sigma x \in \fZ$. Thus $\fZ$ is    $\fS^{\rm big}$-stable.
\end{proof}



\subsection{Discriminants}

Given a commutative ring $B$ and elements $b_1, \ldots, b_n \in B$, we let $\Delta(b_1, \ldots, b_n)=\prod_{1 \le i < j \le n}(b_j-b_i)$ be their discriminant. We put $\Delta_n=\Delta(\xi_1, \ldots, \xi_n)$, which we regard as an element of $R$.

\begin{proposition}
We have
\begin{displaymath}
\sum_{\sigma \in \fS_n/\fS_{n-1}} \sgn(\sigma) \sigma \left( \xi_n^k \cdot \Delta_{n-1} \right) =
\begin{cases}
\Delta_n & \text{if $k=n-1$} \\
0 & \text{if $k<n-1$} \end{cases}
\end{displaymath}
\end{proposition}

\begin{proof}
It suffices to treat the case where $A=\bZ$, so we assume this is the case. Call the sum $F$. Since $F$ is a skew-invariant polynomial in $\xi_1, \ldots, \xi_n$ it is therefore divisible by $\Delta_n$. (Here we have used that~2 is a non-zerodivisor in $A$.) If $k<n-1$ then $\deg(F)<\deg(\Delta_n)$, and so $F=0$. Now suppose $k=n-1$. No monomial appearing in $\Delta_{n-1}$ has an $n-1$ power in it: indeed, in the product description for $\Delta_{n-1}$, the variable $\xi_i$ appears in only $n-2$ factors. It follows that the coefficient of $\xi_n^{n-1}$ in $F$ is $\Delta_{n-1}$, and, in particular, non-zero. Since $\deg(F)=\deg(\Delta_n)$, it follows that $F$ is a non-zero scalar multiple of $\Delta_n$. Comparing the coefficients of $\xi_n^{n-1}$, we see that the scalar is~1.
\end{proof}

\begin{proposition}
\label{prop:discriminant}
Let $\fa$ be a non-zero $\fS$-ideal of $R$. Then $\fa$ contains $c \cdot \Delta_n$ for some $n$ and some non-zero $c \in A$. In fact, if $f \in \fa$ is non-zero then there is some non-zero coefficient $c$ of $f$ such that $c \cdot \Delta_n$ belongs to $\fa$.
\end{proposition}

\begin{proof}
Let $f \in \fa$ be a non-zero element. Suppose $f$ uses the variables $\xi_1, \ldots, \xi_r$, and write
\begin{displaymath}
f  = \sum_{i=0}^d g_i(\xi_1, \ldots, \xi_{r-1}) \xi_r^i
\end{displaymath}
with $g_d \ne 0$. By induction on the number of variables, the $\fS$-ideal generated by $g_d(\xi_1, \ldots, \xi_{r-1})$ contains $c \Delta_m$ for some $m$ and some non-zero $c \in \bk$. Let $X \in A[\fS]$ be such that
\begin{displaymath}
X g_d(\xi_1, \ldots, \xi_{r-1}) = c \Delta(\xi_{a_1}, \ldots, \xi_{a_m})
\end{displaymath}
for some indices $a_1, \ldots, a_m$. We assume, without loss of generality, that $X$ avoids the index $r$, that is, all polynomials appearing in $X$ do not use $\xi_r$, and all permutations appearing in $X$ fix $r$.
	
Now, let $\{b_1, \ldots, b_d\}\subset [\infty]$ be disjoint from $\{1,\ldots,r\}$ and any index appearing in $X$. Let $G \subset \fS$ be the symmetric group on $\{r,b_1,\ldots,b_d\}$, and let $G' \subset G$ be the stabilizer of $r$. Put $Y = \sum_{\sigma \in G/G'} \sgn(\sigma) \sigma$, so that, by the previous lemma, we have
\begin{displaymath}
Y f = g_d(\xi_1, \ldots, \xi_{r-1}) \Delta(\xi_r, \xi_{b_1}, \ldots, \xi_{b_d}).
\end{displaymath}
Therefore,
\begin{displaymath}
XY(f) = c \Delta(\xi_{a_1}, \ldots, \xi_{a_m}) \Delta(\xi_r, \xi_{b_1}, \ldots, \xi_{b_d}).
\end{displaymath}
(Note that $\Delta(\xi_r, \xi_{b_1}, \ldots, \xi_{b_d})$ commutes with $X$ by our choices.) Let $h$ be the product of all linear forms of the form $\xi_{a_i}-\xi_{b_j}$ and $\xi_{a_i}-\xi_r$. Then
\begin{displaymath}
h XY(f) = c \Delta(\xi_{a_1}, \ldots, \xi_{a_m}, \xi_r, \xi_{b_1}, \ldots, \xi_{b_d}).
\end{displaymath}
We thus see (after applying an element of $\fS$) that $c \Delta_n \in \fa$, with $n=m+d+1$.
\end{proof}

\begin{proposition}
\label{prop:extension-contraction}
Let $S=R[(\xi_i-\xi_j)^{-1}]_{i \ne j}$. Then extension and contraction give mutually inverse bijections
\begin{displaymath}
\{ \text{ideals of $A$} \} \leftrightarrow \{ \text{$\fS$-ideals of $S$} \}.
\end{displaymath}
\end{proposition}

\begin{proof}
It is clear that every ideal of $A$ is the contraction of its extension. Conversely, let $\fb$ be an $\fS$-ideal of $S$ and let $\fa$ be its contraction to $A$. Of course, $\fa^e \subset \fb$, where $(-)^e$ denotes extension. Suppose $\fb$ is strictly larger than $\fa$, and let $f \in \fb \setminus \fa^e$. Clearing denominators, we can assume that $f$ belongs to $R$. Furthermore, modifying $f$ by an element of $\fa^e$, we can assume that no coefficient of $f$ belongs to $\fa$. Proposition~\ref{prop:discriminant} then implies that $c \Delta_n$ belongs to $\fb$ for some $n$ where $c$ is some non-zero coefficient of $f$. But $\Delta_n$ is a unit of $S$, and so $c \in \fb \cap A = \fa$, a contradiction.
\end{proof}

\subsection{Finitary points}

Let $x=(w, (x_i)_{i \in [\infty]})$ be a $K$-point of $\fX$. We define the {\bf width} of $x$ to be the cardinality of the subset $\{x_i \mid i \in [\infty] \}$ of $K$, and we say that $x$ is {\bf finitary} if it has finite width. It is clear that width is invariant under field extension, and can thus be defined for points of $\fX$. We let $\fX_{\le n}$ be the set of points of width $\le n$, and $\fX_{\fin}$ the set of points of finitary points. We say that a subset of $\fX$ is {\bf bounded} if it is contained in $\fX_{\le n}$ for some $n$. Let $\fd_n=\Langle \Delta_n \Rangle$ be the ideal of $R$ generated by the $\fS$-orbit of $\Delta_n$.

\begin{proposition}
We have $\fX_{\le n}=V(\fd_{n+1})$. In particular, $\fX_{\le n}$ is Zariski closed.
\end{proposition}

\begin{proof}
It is clear that a $K$-point $x$ has width $\le n$ if and only if $\Delta_{n+1}(\sigma x)=0$ for all $\sigma \in \fS$, and so the result follows.
\end{proof}

\begin{proposition}
\label{prop:decomposition-fX}
Suppose that $A$ is noetherian. Then every Zariski closed $\fS$-subset $\fZ$ of $\fX=\fX_W$ has the form $\fZ_1 \cup \fZ_2$ where $\fZ_1=\fX_{W'}$ for some closed subset $W'$ of $W$ and $\fZ_2$ is a closed and bounded $\fS$-subset of $\fX_W$. 
\end{proposition}

\begin{proof}
We proceed by noetherian induction on $W$; thus we assume the result holds whenever $W$ is replaced by a proper closed subset. Let $\fZ$ be given, and let $I \subset R$ be its radical ideal which we may assume is not nilpotent. Then, by Proposition~\ref{prop:discriminant}, $I$ contains $a \Delta_n$ for some non-nilpotent $a \in A$ and some $n$. We thus have
\begin{displaymath}
\fZ=(\fZ \cap V(a)) \cup (\fZ \cap \fX_{\le n}).
\end{displaymath}
Now, $V(a)$ is a proper closed subset of $W$. Thus, by the inductive hypothesis, we have $\fZ \cap V(a) = \fZ_1 \cup \fZ_2'$ where $\fZ_1=\fX_{W'}$ for some closed subset $W'$ of $W$, and $\fZ_2'$ is a closed and bounded $\fS$-subset of $\fX$. Of course, $\fZ_2''=\fZ \cap \fX_W^{\le n}$ is also a closed and bounded $\fS$-subset of $\fX$. We thus have $\fZ=\fZ_1 \cup \fZ_2$, where $\fZ_2=\fZ_2' \cup \fZ_2''$, as required.
\end{proof}

\subsection{The type of a point} \label{ss:type}

We now introduce an invariant that is finer than width. Recall that an $\infty$-partition is a sequence $\lambda=(\lambda_1, \lambda_2, \ldots)$ such that $\lambda_1=\infty$, $\lambda_i \ge \lambda_{i+1}$, and $\lambda_i=0$ for $i \gg 0$. Let $x=(w,(x_i)_{i \in [\infty]})$ be a finitary $K$-point of $\fX$ of width $n$. Let $a_1, \ldots, a_n$ be the $n$ values that the $x_i$ assume, and let $\cU_k \subset [\infty]$ be the set of indices $i$ such that $x_i=a_k$. Permuting the $a$'s if necessary, we assume that $\# \cU_k \ge \# \cU_{k+1}$. Then putting $\lambda_k=\# \cU_k$ for $1 \le k \le n$, and $\lambda_k=0$ for $k>n$, we see that $\lambda$ is an $\infty$-partition. We define the {\bf type} of $x$ to be $\lambda$. This quantity is clear invariant under field extension, and can thus be associated to fintary points of $\fX$. It is also clearly invariant under the action of the big symmetric group.

\section{The space $\cX$}

\subsection{Affine spaces}

For a finite set $S$, we let $\bA^S_W$ be the affine space over the base scheme $W$ with coordinates indexed by the set $S$. We let $\bA^{[S]}_W$ be the open subscheme of $\bA^S_W$ where the coordinates are all distinct. Given a function $f \colon S \to T$ of finite sets, we let $\alpha_f \colon \bA_W^T \to \bA_W^S$ be the map that takes $(x_j)_{j \in T}$ to $(x_{f(i)})_{i \in S}$. We note that if $f$ is an injection then $\alpha_f$ is a projection, while if $f$ is a surjection then $\alpha_f$ is a multi-diagonal map (and thus a closed immersion).

\subsection{Generalized partitions and compositions} \label{ss:part}

A {\bf generalized partition} is a tuple $\lambda=(\lambda_1,\lambda_2,\ldots)$ where $\lambda_i \in \bN \cup \{\infty\}$ such that $\lambda_i \ge \lambda_{i+1}$ for all $i$ and $\lambda_i=0$ for $i \gg 0$. We define length of $\lambda$, denoted $\ell(\lambda)$, to be the largest $i$ such that $\lambda_i >0$. A generalized partition $\lambda$ is an $\infty$-partition if $\lambda_1 = \infty$, otherwise $\lambda$ is called a {\bf finite partition}.  We denote the set of generalized partition by $\wt{\Lambda}$, and the set of partitions of infinity by $\Lambda$.

A {\bf generalized composition} is a pair $(S, \lambda)$ consisting of a finite set $S$ and a function $\lambda \colon S \to \bZ_{>0} \cup \{\infty\}$. We typically simply write $\lambda$ for the generalized composition, and denote the index set $S$ by $\langle \lambda \rangle$. Given a generalized composition $\lambda$, we let $\vert \lambda \vert = \sum_{i \in \langle \lambda \rangle} \lambda_i$ and $\ell(\lambda)=\# \langle \lambda \rangle$. We say that $\lambda$ is a {\bf composition of $\infty$} or an {\bf $\infty$-composition} if $\vert \lambda \vert=\infty$, or equivalently, if $\lambda_i=\infty$ for some $i \in \langle \lambda \rangle$. We regard any generalized partition $\lambda$ as a generalized composition on the index set $[\ell(\lambda)]$.

Let $\lambda$ and $\mu$ be generalized compositions. A {\bf map} $f \colon \lambda \to \mu$ is a function $f \colon \langle \lambda \rangle \to \langle \mu \rangle$ such that $f_*(\lambda) \le \mu$, i.e., for all $j \in \langle \mu \rangle$ we have $\sum_{i \in f^{-1}(j)} \lambda_i \le \mu_j$. We say that $f$ is a {\bf principal surjection} if $f_*(\lambda)=\mu$, which implies that $f$ is a surjection. We write $f \colon \lambda \psurj \mu$ to indicate that $f$ is a principal surjection. Every map $f$ can be factored as $g \circ h$ where $h$ is a principal surjection and $g$ is an injection. The composition of two maps is again a map, and so we have a category of generalized compositions.

Let $\lambda$ and $\mu$ be generalized partitions. We define $\lambda \le \mu$ if $\lambda$ can be obtained from $\mu$ by decreasing (or removing) parts, and we define $\lambda \preceq \mu$ if $\lambda$ can be obtained from $\mu$ by combining and decreasing (or removing) parts. We note here that if $f \colon \lambda \to \mu$ is an injection then $\lambda \le \mu$ and if $f$ is a principal surjection then $\mu \preceq \lambda$. A poset is called a {\bf well-quasi-order (wqo)} if it satisfies the descending chain condition and has no infinite antichains. We refer to \cite[\S 2]{catgb} for general facts about wqo's, though note there the term ``noetherian'' is used in place of ``well-quasi-order.''

\begin{proposition}
	\label{prop:min-are-finite}
	The posets $(\wt{\Lambda}, \le)$ and $(\wt{\Lambda}, \preceq)$ are wqo's. 
\end{proposition}
\begin{proof}
	Since $\bN \cup \{\infty\}$ is a wqo, we conclude that the poset of sequences on alphabet $\bN \cup \{\infty\}$ is also a wqo (Higman's lemma, or see \cite[Theorem~2.5]{catgb}). Thus $(\wt{\Lambda}, \le)$, being a subposet of this poset of sequences, is also a wqo.  Since $\lambda \le \mu \implies \lambda \preceq \mu$, the poset $(\wt{\Lambda}, \preceq)$ does not have infinite antichains. Now suppose $\lambda_1 \succeq \lambda_2 \succeq \cdots$ is a decreasing chain. Then $\ell(\lambda_{i+1}) \le \ell(\lambda_i)$ for each $i$. So we may as well assume that $\ell(\lambda_i)$ is independent of $i$. But then $\lambda_i \ge \lambda_{i+1}$ for all $i$. Since $(\wt{\Lambda}, \le)$ is a wqo, this chain stabilizes. This completes the proof. 
\end{proof}

\subsection{The space $\cX$ and its subspaces} \label{ss:subspaces}

For an $\infty$-composition $\lambda$, we put $\cX_{W,\lambda}=\bA^{\langle \lambda \rangle}_W$ and $\cX_{W,[\lambda]}=\bA^{[\langle \lambda \rangle]}_W$. We regard $\cX_W=\{\cX_{W,\lambda}\}_{\lambda}$ as a single object. We omit the $W$ from the notation when possible.

By a ``subset'' of $\cX$ we mean a rule $\cZ$ assigning to every $\infty$-composition $\lambda$ a subset $\cZ_{\lambda}$ of $\cX_{\lambda}$; here we simply regard $\cX_{\lambda}$ as set of points in its underlying topological space (including its non-closed points). We consider several conditions on such a subset $\cZ$:
\begin{itemize}
\item We say that $\cZ$ is a {\bf C1 subset} if for every principal surjection $f \colon \lambda \psurj \mu$ of $\infty$-compositions we have $\alpha_f^{-1}(\cZ_{\lambda})=\cZ_{\mu}$.
\item We say that a point $x \in \cX_{\lambda}$ is {\bf $N$-approximable} by $\cZ$, for a positive integer $N$, if there exists an $\infty$-composition $\mu$ and an injective function $f \colon \langle \lambda \rangle \to \langle \mu \rangle$ such that $\mu_{f(i)} \ge \min(\lambda_i, N)$ for all $i \in \langle \lambda \rangle$ and $x$ belongs to the image of $\cZ_{\mu}$ under the projection map $\alpha_f \colon \cX_{\mu} \to \cX_{\lambda}$. (Note: $f$ need not define a map of $\infty$-compositions $\lambda \to \mu$.) We say that $x$ is {\bf approximable} by $\cZ$ if it is $N$-approximable for all $N$. We say that $\cZ$ is a {\bf C2 subset} if it is C1 and contains all of its approximable points (i.e., whenever $x \in \cX_{\lambda}$ is approximable by $\cZ$ we have $x \in \cZ_{\lambda}$).
\item We say that $\cZ$ is a {\bf C3 subvariety} if it is C2 and $\cZ_{\lambda}$ is Zariski closed in $\cX_{\lambda}$ for all $\lambda$.
\end{itemize}
We note that if $\cZ$ is a C1 subset then for any isomorphism $f \colon \lambda \to \mu$ of $\infty$-compositions we have $\alpha_f(\cZ_{\mu})=\cZ_{\lambda}$, since an isomorphism is a principal surjection. Thus $\cZ$ is determined by $\cZ_{\lambda}$ with $\lambda$ an $\infty$-partition. Ultimately, we are most interested in C3 subvarieties, but C1 and C2 subsets will be helpful intermediate objects.

\begin{remark}
Let us try to motivate the above conditions. Suppose that $\fZ$ is a Zariski closed $\fS$-stable subset of $\fX$, and let $\cZ=\Phi(\fZ)$ as defined in \S \ref{ss:descr}. Then $\cZ_{(\infty)} \subset \bA^1$ consists of all points $a$ such that $(a,a,a,\ldots) \in \fZ$, while $\cZ_{(\infty,1)} \subset \bA^2$ consists of all points $(a,b)$ such that $(b,a,a,a,\ldots) \in \fZ$.

It is clear that $\cZ_{(\infty)}$ is simply the pullback of $\cZ_{(\infty,1)}$ along the diagonal map $\bA^1 \to \bA^2$. This is a special case of the C1 condition. The general case is not much different (see Proposition~\ref{prop:C1-corr}), and so $\cZ$ is a C1 subset.

Suppose now that $(a,b) \in \cZ_{(\infty,1)}$. Then $a \in \cX_{(\infty)}$ is approximable by $\cZ$: indeed, if $f \colon (\infty) \to (\infty,1)$ denotes the natural inclusion then $a=\alpha_f(b,a)$. We have already seen (in \S \ref{ss:descr}) that $a$ belongs to $\cZ_{(\infty)}$ in this case. This verifies the C2 condition in a special case. Again, the general case is not much more difficult (see Proposition~\ref{prop:C2-corr}), and so $\cZ$ is a C2 subset.

By definition, $\cZ_{\lambda}$ is in the inverse image of $\fZ$ under the multi-diagonal map $\cX_{\lambda} \to \fX$ corresponding to some parition $\cU$ of $[\infty]$ with $\# \cU_i=\lambda_i$. Since the multi-diagonal map is continuous and $\fZ$ is Zariski closed, it follows that $\cZ_{\lambda}$ is Zariski closed in $\cX_{\lambda}$. Thus $\cZ$ is a C3 subvariety.
\end{remark}

\subsection{First properties of subspaces of $\cX$}

We now establish some simple properties of C1, C2, and C3 subsets of $\cX$. For a subset $\cZ$ of $\cX$, we let $\cZ_{[\lambda]}=\cZ_{\lambda} \cap \cX_{[\lambda]}$.

\begin{proposition} \label{prop:C1-contain}
Let $\cZ$ and $\cZ'$ be two C1 subsets of $\cX$. Then $\cZ \subset \cZ'$ if and only if $\cZ_{[\lambda]}\subset \cZ'_{[\lambda]}$ for all $\infty$-compositions $\lambda$.
\end{proposition}

\begin{proof}
If $\cZ \subset \cZ'$ then obviously $\cZ_{[\lambda]} \subset \cZ'_{[\lambda]}$ for all $\lambda$. Thus suppose $\cZ_{[\lambda]} \subset \cZ'_{[\lambda]}$ for all $\lambda$. Let $\lambda$ be an $\infty$-composition and let $x \in \cZ_{\lambda}$. Define an equivalence relation $\sim$ on $\langle \lambda \rangle$ by $i \sim j$ if $x_i=x_j$. Let $\langle \mu \rangle=\langle \lambda \rangle/\sim$, let $f \colon \langle \lambda \rangle \to \langle \mu \rangle$ be the quotient map, and let $\mu=f_*(\lambda)$; thus $f \colon \lambda \psurj \mu$ is a principal surjection. Then $x=\alpha_f(y)$ for some $y \in \cX_{[\mu]}$. Since $y \in \alpha_f^{-1}(\cZ_{\lambda})$ and $\cZ$ is C1, we have $y \in \cZ_{[\mu]}$. Thus $y \in \cZ'_{[\mu]}$ by our hypothesis. Therefore $x \in \alpha_f(\cZ'_{\mu}) \subset \cZ'_{\lambda}$ since $\cZ'$ is C1, and so $\cZ \subset \cZ'$.
\end{proof}

\begin{proposition} \label{prop:C2-map}
Let $\cZ$ be a C2 subset of $\cX$ and let $f \colon \lambda \to \mu$ be a map of $\infty$-compositions. Then $\alpha_f(\cZ_{\mu}) \subset \cZ_{\lambda}$.
\end{proposition}

\begin{proof}
First suppose that $f$ is an injective map, and let $x \in \cZ_{\mu}$. Then $\alpha_f(x) \in \cX_{\lambda}$ is approximable by $\cZ$: indeed, this is witnessed by the point $x \in \cZ_{\mu}$. Since $\cZ$ is C2, it follows that $\alpha_f(x) \in \cZ_{\lambda}$. If $f$ is a principal surjection, then $\alpha_f(\cZ_{\mu}) \subset \cZ_{\lambda}$ by the C1 property. The general case now follows.
\end{proof}

\begin{proposition} \label{prop:C2-descending}
Let $\cZ$ be a C2 subset of $\cX$, and let $\lambda$ and $\mu$ be $\infty$-compositions with $\langle \lambda \rangle = \langle \mu \rangle$ and $\lambda_i \le \mu_i$ for all $i \in \langle \lambda \rangle$. Then $\cZ_{\mu} \subset \cZ_{\lambda}$.
\end{proposition}

\begin{proof}
The identity map $f \colon \lambda \to \mu$ is a map of $\infty$-compositions, and so $\cZ_{\mu}=\alpha_f(\cZ_{\mu}) \subset \cZ_{\lambda}$ by Proposition~\ref{prop:C2-map}.
\end{proof}

\begin{proposition}
	\label{prop:continuity}
	Let $\cZ$ be a C2 subset of $\cX$, and let $\lambda^1 \le \lambda^2 \le \ldots$ be $\infty$-partitions, all of length $r$, with limit $\lambda$.   Then $\cZ_{\lambda} = \bigcap_{i \ge 1} \cZ_{\lambda^i}$.
\end{proposition}
\begin{proof}
	By the previous proposition, we have $\cZ_{\lambda} \subset \bigcap_{i \ge 1} \cZ_{\lambda^i}$. Conversely, let $x \in \cZ_{\lambda^i}$ for each $i$. Then $x \in \cX_{\lambda}$, and for each $N$ we may find an $n$ large enough such that $\lambda^n_i \ge \min(\lambda_i, N)$ for all $i \in [r]$. Thus $x$ is $N$-approximable for each $N$, and so $x \in \cZ_{\lambda}$.
\end{proof}

\section{Constructing subvarieties of $\cX$} \label{s:construct}

\subsection{Correspondences}

Let $\lambda$ and $\mu$ be generalized compositions. A {\bf correspondence} $f \colon \lambda \dashrightarrow \mu$ is a pair $(f_1, f_2)$ where $f_1 \colon \rho \psurj \lambda$ is a principal surjection and $f_2 \colon \rho \to \mu$ is an arbitrary map; here $\rho$ denotes another generalized composition. Suppose that $\lambda$ and $\mu$ are $\infty$-compositions; note that $\rho$ is then as well. For $S \subset \cX_{\mu}$, we let $\alpha_f(S)=\alpha_{f_1}^{-1}(\alpha_{f_2}(S))$, a subset of $\cX_{\lambda}$. Note that $x \in \alpha_f(S)$ if and only if $\alpha_{f_1}(x) \in \alpha_{f_2}(S)$. We say that a subset $\cZ$ of $\cX$ is {\bf closed under correspondences} if $\alpha_f(\cZ_{\mu}) \subset \cZ_{\lambda}$ for all correspondences $f \colon \lambda \dashrightarrow \mu$.

\begin{proposition} \label{prop:corr-closed}
We have the following implications for subsets of $\cX$:
\begin{displaymath}
\text{C2} \implies \text{closed under correspondences} \implies \text{C1}.
\end{displaymath}
\end{proposition}

\begin{proof}
Suppose $\cZ$ is C2. Let $f \colon \lambda \dashrightarrow \mu$ be a correspondence of $\infty$-compositions, given by $f=(f_1 \colon \rho \to \lambda, f_2 \colon \rho \to \mu)$. Then $\alpha_f(\cZ_{\mu})=\alpha_{f_1}^{-1}(\alpha_{f_2}(\cZ_{\mu}))$. By Proposition~\ref{prop:C2-map}, we have $\alpha_{f_2}(\cZ_{\mu}) \subset \cZ_{\rho}$, and by the C1 property we have $\alpha_{f_1}^{-1}(\cZ_{\rho}) \subset \cZ_{\lambda}$. Thus $\cZ$ is closed under correspondences.

Now suppose that $\cZ$ is closed under correspondences, and let $f \colon \lambda \psurj \mu$ be a principal surjection. Then $\alpha_f(\cZ_{\mu}) \subset \cZ_{\lambda}$ by assumption. Let $g \colon \mu \dashrightarrow \lambda$ be the correspondence $(f, \id_{\lambda})$. Again $\alpha_f^{-1}(\cZ_{\lambda})=\alpha_g(\cZ_{\lambda}) \subset \cZ_{\mu}$ by assumption. Thus $\alpha_f^{-1}(\cZ_{\lambda})=\cZ{\mu}$, and so $\cZ$ is C1.
\end{proof}

\subsection{Composing correspondences}


In general, composition of correspondences is carried out using fiber products. Unfortunately, the category of $\infty$-compositions does not have fiber products (see Example~\ref{ex:no-fiber-product}). However, the following proposition provides us with an approximate fiber product that will be good enough.

\begin{proposition} \label{prop:pullback}
For $i=1,2$, let  $f_i \colon \mu^i \to \mu$ be a map of $\infty$-weightings. Then there exists an $\infty$-composition $\wt{\mu}$ and maps $g_i \colon \wt{\mu} \to \mu^i$ of $\infty$-compositions such that the following diagram commutes:
\begin{displaymath}
\xymatrix{
\wt{\mu} \ar[r]^{g_2} \ar[d]_{g_1} & \mu^2 \ar[d]^{f_2} \\
\mu^1 \ar[r]^{f_1} & \mu.
}
\end{displaymath}
Moreover, the following holds:
\begin{enumerate}
\item If $f_2$ is a principal surjection then so is $g_1$.
\item If $f_2$ is an injection then so is $g_1$. 
\end{enumerate}
\end{proposition}

\begin{proof}
Note that we can prove the result fiberwise, in other words, it is enough to deal with the case when $\mu$ is a singleton (but possibly with finite nonzero weight). So let $\mu = (\mu_1)$. Suppose $\mu^i = (\mu^i_1, \mu^i _2, \ldots, \mu^i _{r_i})$ where $\mu^i_j$ is a non-increasing sequence with values in $\bN \cup \{\infty\}$. We now construct $\wt{\mu}$, $g_1$ and $g_2$ by induction on the lengths $r_1, r_2$ of the $\infty$-compositions $\mu^1, \mu^2$. Set $\wt{\mu}_1 = \min\{\mu^1_1, \mu^2_1 \}$. The base case is when $\wt{\mu}_1  = \mu^i_1$ for some $i \in \{1,2\}$ and $\mu^i$ has only one part. In this case, we take $\wt{\mu} = \mu^i$. In general, by induction, there exists a diagram
\begin{displaymath}
\xymatrix{
	\wt{\lambda} \ar[r]^{g'_2} \ar[d]_{g'_1} & \lambda^2 \ar[d]^{f'_2} \\
	\lambda^1 \ar[r]^{f'_1} & \lambda,
}
\end{displaymath}
where $\lambda = (\mu_1 - \wt{\mu}_1)$, $\lambda^i = (\mu^i_1 - \wt{\mu}_1, \mu^i _2, \ldots, \mu^i _{r_i})$ (where the elements may need to be rearranged to make them non-increasing), and $f_1', f_2'$ are obtained from restricting $f_1, f_2$. Now set $\wt{\mu} = (\wt{\mu}_1, \wt{\lambda})$. Then $g_i'$ can be extended to $g_i \colon \wt{\mu} \to \mu^i$ by sending $\wt{\mu}_1$ to $\mu^i_1$. It is clear that $f_2g_2 = f_1g_1$, which finishes the construction of $\wt{\mu}$, $g_1$ and $g_2$.  

For (a), assume that $f_2$ is a principal surjection. In the base case, we must have $\wt{\mu} =  \mu^1$ and $g_1 = \id$, so the result follows. In the inductive case, we see that $f_2'$ is a principal surjection. By induction $g_1'$ is a principal surjection. It follows that $g_2$ is a principal surjection.

For (b), assume that $f_2$ is an injection. In the base case, either $\wt{\mu} = \mu^1$ and $\mu^1$ has only one part, or $\wt{\mu} = \mu^2$ and $\mu^1_1 \ge \mu^2_1$. In both cases $g_1$ is an injection. In the inductive case, note that  $f_2'$ is an injection. By induction, $g_1'$ is an injection. Moreover, in this case we must have $\wt{\mu}_1 = \mu^1_1$ (as we are not in the base case), and so $\lambda^1$ has $r_1 - 1$ parts. It follows that $g_1$ is an injection. This finishes the proof.
\end{proof}

Let $f \colon \lambda \dashrightarrow \mu$ and $g \colon \mu \dashrightarrow \nu$ be correspondences, given by data $(f_1 \colon \rho \psurj \lambda, f_2 \colon \rho \to \mu)$ and $(g_1 \colon \sigma \psurj \mu, g_2 \colon \sigma \to \nu)$. We say that a correspondence $h \colon \lambda \dashrightarrow \nu$, given by data $(h_1 \colon \tau \to \lambda, \tau \to \nu)$ is a {\bf composition} of $f$ and $g$ if there exists a commutative diagram
\begin{displaymath}
\xymatrix{
&& \tau \ar@{>..>}[ld]_{h_1'} \ar@{..>}[rd]^{h_2'}\\
& \rho \ar@{>->}[ld]_{f_1} \ar[rd]^{f_2} && \sigma \ar@{>->}[ld]_{g_1} \ar[rd]^{g_2} \\
\lambda && \mu && \nu
}
\end{displaymath}
such that $h_1'$ is a principal surjection, $h_1=f_1 \circ h_1'$, and $h_2=h_2 \circ h_2'$. We note that a composition need not be unique.

\begin{proposition} \label{prop:comp-corr}
Let $f \colon \lambda \dashrightarrow \mu$ and $g \colon \mu \dashrightarrow \nu$ be correspondences. Then a composition exists. Moreover, for any composition $h$ and any $S \subset \cX_{\nu}$, we have $\alpha_f(\alpha_g(S)) \subset \alpha_h(S)$.
\end{proposition}

\begin{proof}
Existence follows from Proposition~\ref{prop:pullback}. Now suppose that $h$ is any composition of $f$ and $g$. Use notation as in the previous paragraph. Suppose $x \in \alpha_f(\alpha_g(S))$. By definition, this means $\alpha_{f_1}(x)=\alpha_{f_2}(y)$ for some $y \in \alpha_g(S)$; again, by definition, we have $\alpha_{g_1}(y) \in \alpha_{g_2}(S)$. We thus see that
\begin{displaymath}
\alpha_{h_1}(x)=\alpha_{h_1'}(\alpha_{f_1}(x))=\alpha_{h_1'}(\alpha_{f_2}(y))=\alpha_{h_2'}(\alpha_{g_1}(y)) \in \alpha_{h_2'}(\alpha_{g_2}(S)) = \alpha_{h_2}(S),
\end{displaymath}
and so $x \in \alpha_h(S)$.
\end{proof}

\begin{example} \label{ex:no-fiber-product}
Fiber products do not exist in the category of $\infty$-compositions, in general. For example, let $\mu = (\infty, 8), \wt{\mu}^1 = (\infty, 6,2), \wt{\mu}^2 = (\infty, 4,4)$, and let $\pi_i \colon  \wt{\mu}^i \psurj \mu$ be the unique principal surjections. Then the natural candidate for $\wt{\mu}^1 \times_{\mu} \wt{\mu}^2$ is $(\infty, 4,2,1,1)$. But there are two maps  $(\infty, 4,2,1, 1) \to \wt{\mu}^1$ -- one that combines $4$ and $2$ together, and the second one that combines $4$ and the two $1$s together. These two maps do not factor through each other and so the fiber product  $\wt{\mu}^1 \times_{\mu} \wt{\mu}^2$ doesn't exist.
\end{example}

\subsection{Construction of C2 subsets}

Let $\lambda$ be a $\infty$-composition and let $Z$ be a subset of $\cX_{\lambda}$. Define a subset $\Gamma_{\lambda}^{\circ}(Z)$ of $\cX$ by
\begin{displaymath}
\Gamma^{\circ}_{\lambda}(Z)_{\mu} = \bigcup_{f \colon \mu \dashrightarrow \lambda} \alpha_f(Z),
\end{displaymath}
where the union is over all correspondences $f$.

\begin{theorem} \label{thm:gamma0}
Let $Z$ be a subset of $\cX_{\lambda}$. Then $\cZ=\Gamma^{\circ}_{\lambda}(Z)$ is a C2 subset of $\cX$. Moreover, if $\cZ'$ is any C2 subset of $\cX$ such that $\cZ'_{\lambda}$ contains $Z$ then $\cZ'$ contains $\cZ$.
\end{theorem}

\begin{proof}
We first claim that $\cZ$ is closed under the action of correspondences. Indeed, suppose $x \in \cZ_{\nu}$ and $f \colon \mu \dashrightarrow \nu$ is a correspondence. Then, by definition of $\cZ_{\nu}$, there exists a correspondence $g \colon \nu \dashrightarrow \lambda$ such that $x \in \alpha_g(Z)$. Let $h \colon \mu \dashrightarrow \lambda$ be a composition of $f$ and $g$. Then
\begin{displaymath}
\alpha_f(x) \in \alpha_f(\alpha_g(Z)) \subset \alpha_h(Z) \subset \cZ_{\mu},
\end{displaymath}
where we have used Proposition~\ref{prop:comp-corr}. This proves the claim. In particular, $\cZ$ is C1 by  Proposition~\ref{prop:corr-closed}.

We now claim that $\cZ$ is a C2 subset of $\cX$. Without loss of generality, suppose $\langle \lambda \rangle=[n]$ and that $\lambda_1, \ldots, \lambda_r$ are infinite and $\lambda_{r+1}, \ldots, \lambda_n$ are finite. Suppose $x \in \cX_{\mu}$ is approximable by $\cZ$. Let $N$ be larger than $\lambda_{r+1}+\cdots+\lambda_n$ and all finite parts of $\mu$. Since $x$ is $N$-approximable by $\cZ$, we can find an $\infty$-composition $\nu$ with $\langle \mu \rangle \subset \langle \nu \rangle$ and $\nu_i \ge \min(N, \mu_i)$ for all $i \in \langle \mu \rangle$ such that $x$ is the image of some point $y \in \cZ_{\nu}$ under the projection map $\cX_{\nu} \to \cX_{\mu}$. Now, since $y \in \cZ_{\nu}$, there is some correspondence $f \colon \nu \dashrightarrow \lambda$ such that $y \in \alpha_f(Z)$. Let $f=(f_1 \colon \rho \psurj \nu, f_2 \colon \rho \to \lambda)$ where $f_1$ is a principal surjection.

Suppose $\nu_i \ge N$. We claim that there is some $j \in f_1^{-1}(i)$ such that $f_2(j) \le r$. Indeed, suppose this were not the case. Then $f_2$ would map $f_1^{-1}(i)$ into $\{r+1, \ldots, n\}$, and so $f_1^{-1}(i)$ would be contained in $\bigcup_{k=r+1}^n f_2^{-1}(k)$. But the total weight of the set $f_2^{-1}(k)$ is at most $\lambda_k$, and so the total weight of the set $\bigcup_{k=r+1}^n f_2^{-1}(k)$ is $\lambda_{r+1}+\cdots+\lambda_n<N$. However, the total weight of the set $f_1^{-1}(i)$ is $\nu_i$, which is $\ge N$, a contradiction.

For each $i \in \langle \nu \rangle$ with $\nu_i \ge N$ pick $a(i) \in f_1^{-1}(i)$ such that $f_2(a(i)) \le r$. Let $\ol{\nu}$ be the $\infty$-composition obtained from $\nu$ by changing any entry that is $\ge N$ to to $\infty$; similarly, let $\ol{\rho}$ be the $\infty$-composition obtained from $\rho$ by changing the entries at $a(1), \ldots, a(r)$ to $\infty$. Then the function $f_1 \colon \langle \rho \rangle \to \langle \nu \rangle$ defines a principal surjection $g_1 \colon \ol{\rho} \psurj \ol{\nu}$ and the function $f_2 \colon \langle \rho \rangle \to [n]$ defines a map of $\infty$-compositions $g_2 \colon \ol{\rho} \to \lambda$. Let $g \colon \ol{\nu} \dashrightarrow \lambda$ be the correspondence $(g_1, g_2)$. Then $y \in \alpha_g(Z)$, and so $y \in \cZ_{\ol{\nu}}$. (Note that $\alpha_{g_1}$ and $\alpha_{g_2}$ are the same maps between affines spaces as $\alpha_{f_1}$ and $\alpha_{f_2}$.)

The standard inclusion $\langle \mu \rangle \to \langle \nu \rangle$ defines a map of $\infty$-compositions $\mu \to \ol{\nu}$. Thus by the first paragraph, the projection map $\cX_{\ol{\nu}} \to \cX_{\mu}$ maps $\cZ_{\ol{\nu}}$ into $\cZ_{\lambda}$. Since the image of $y \in \cZ_{\ol{\nu}}$ under this map is $x$, we see that $x \in \cZ_{\mu}$. Thus $\cZ$ is C2, as claimed.

Finally, suppose $\cZ'$ is a second C2 subset of $\cX$ such that $\cZ'_{\lambda}$ contains $Z$. Since $\cZ'$ is closed under the action of correspondences (Proposition~\ref{prop:corr-closed}), we thus have $\cZ \subset \cZ'$.
\end{proof}

We need two additional properties of the $\Gamma_{\lambda}^{\circ}$ construction. We say that a correspondence $f \colon \mu \dashrightarrow \lambda$ given by  $f = (f_1 \colon \rho \psurj \mu, f_2 \colon \rho \to \lambda)$ is {\bf good} if for each $i \in \langle \mu \rangle$ the size of $f_1^{-1}(i)$ is bounded by the number of parts in $\lambda$, and if $\mu_i$ is larger than the sum of finite parts of $\lambda$ then the fiber $f_1^{-1}(i)$ is a singleton. We note that the number of good correspondences with fixed domain and target is finite.

\begin{proposition} \label{prop:finite-union}
Let $Z$ be a subset of $\cX_{\lambda}$. Then
\begin{displaymath}
\Gamma_{\lambda}^{\circ}(Z)_{\mu} =  \bigcup_{\substack{f \colon \mu \dashrightarrow \lambda \\ \text{$f$ is good}}} \alpha_f(Z).
\end{displaymath}
\end{proposition}

\begin{proof}
Let $f \colon \mu \dashrightarrow \lambda$ be an arbitrary correspondence, given by $f=(f_1 \colon \rho \psurj \mu, f_2 \colon \rho \to \lambda)$ with $f_1$ a principal surjection. Define an equivalence relation $\sim$ on $\langle \rho \rangle$ by $i \sim j$ if $f_1(i)=f_1(j)$ and $f_2(i)=f_2(j)$. Let $\langle \tau \rangle = \langle \rho \rangle/\sim$, let $\pi \colon \langle \rho \rangle \to \langle \tau \rangle$ be the quotient map, and let $\tau=\pi_*(\rho)$, so that $\pi \colon \rho \psurj \tau$ is a principal surjection. Write $f_1=g_1 \circ \pi$ and $f_2=g_2 \circ \pi$. Then we have a correspondence $g \colon \mu \dashrightarrow \lambda$ given by $(g_1 \colon \tau \psurj \mu, g_2 \colon \tau \to \lambda)$. One easily sees that $\alpha_f(Z)=\alpha_g(Z)$. Note that $\# g_1^{-1}(i) \le \ell(\lambda)$ for all $i \in [\ell(\mu)]$.

Let $e$ be the sum of finite parts of $\lambda$. For each $i \in \langle \mu \rangle$ with $\mu_i >e$, choose $a(i) \in g_1^{-1}(i)$ such that $\lambda_{g_2(a(i))} = \infty$, which must exist. Let $\sigma$ be the composition on the underlying set $\langle \sigma \rangle \coloneq \{ a(i) \in \langle \mu \rangle \colon \mu_i >e  \} \sqcup \{ i \in \langle \mu \rangle \colon \mu_i \le e  \}$ where $\sigma_{a(i)} = \sum_{j \in g_1^{-1}(i)} \tau_j$ for each $i$ with $\mu_i >e$. Let $h_1 \colon \sigma \psurj \mu$  and $h_2 \colon \sigma \to \rho$ be the induced maps. Then we have a correspondence $h \colon \mu \dashrightarrow \lambda$ given by $(h_1, h_2)$, which is good. Furthermore, it is clear that $\alpha_g(Z) \subset \alpha_h(Z)$.

We have thus shown that for every correspondence $f \colon \mu \dashrightarrow \lambda$ there is a good correspondence $h \colon \mu \dashrightarrow \lambda$ such that $\alpha_f(Z) \subset \alpha_h(Z)$. The result follows.
\end{proof}

\begin{proposition} \label{prop:gamma-aux3}
Let $Z \subset \cX_{\lambda}$ and $\cZ=\Gamma^{\circ}_{\lambda}(Z)$. Then
\begin{displaymath}
\cZ_{[\lambda]} = \bigcup_{f \in \Aut(\lambda)} \alpha_f(Z \cap \cX_{[\lambda]}).
\end{displaymath}
\end{proposition}

\begin{proof}
Let $Z'=Z \cap \cX_{[\lambda]}$. It is clear that $\cZ$ contains $\alpha_f(Z')$ for each automorphism $f$ of $\lambda$. Conversely, suppose that $x \in \cZ_{[\lambda]}$. Then $x \in \alpha_g(Z)$ for some correspondence $g \colon \lambda \dashrightarrow \lambda$. Write $g=(g_1 \colon \rho \psurj \lambda, g_2 \colon \rho \to \lambda)$. Then $\alpha_{g_1}(x)=\alpha_{g_2}(y)$ for some $y \in Z$. Each coordinate of $x$ must appear as a coordinate of $y$. Since $x$ has distinct coordinates, the same is true for $y$. Thus $y \in Z'$ and there is a permutation $f$ of $\langle \lambda \rangle$ such that $x_i=y_{f(i)}$ for each $i$. We claim that $f$ is an automorphism of the $\infty$-composition $\lambda$. Let $i \in \langle \lambda \rangle$. Given any $j \in g_1^{-1}(i)$, we have $x_i=y_{g_2(i)}$, and so $g_2(i)=f(i)$. It follows that $g_2$ maps all of $g_1^{-1}(i)$ to $f(i)$. Since $g_1^{-1}(i)$ has total weight $\lambda_i$ (as $g_1$ is a principal surjection), it follows that $\lambda_i \le \lambda_{f(i)}$ (since $g_2$ is a map of $\infty$-compositions). Since $f$ is bijective, we must have $\lambda_i=\lambda_{f(i)}$ for all $i$, and so $f$ is an automorphism of $\lambda$. This completes the proof.
\end{proof}

\subsection{$\End$-stable subsets}

Let $\lambda$ be an $\infty$-composition. We say that a set $Z \subset \cX_{\lambda}$ is {\bf $\End(\lambda)$-stable} if $\alpha_f(Z) \subset Z$ for every map $f \colon \lambda \to \lambda$ of $\infty$-compositions. We define $Z^e$ to be the Zariski closure of the set $\bigcup_{f \in \End(\lambda)} \alpha_f(Z)$. It is clear that $Z^e$ is $\End(\lambda)$-stable. Moreover, if $Z$ is $\End(\lambda)$-stable and Zariski closed then $Z^e=Z$. In particular, we see that $(Z^e)^e=Z^e$ for any $Z$.

%

\begin{proposition} \label{prop:corr-preserves-closure}
Suppose $Z$ is $\End(\lambda)$-stable Zariski closed subset of $\cX_{\lambda}$. Let $f \colon \mu \dashrightarrow \lambda$ be a correspondence of $\infty$-compositions. Then $\alpha_f(Z)$ is Zariski closed in $\cX_{\mu}$.
\end{proposition}

\begin{proof}
First suppose that $f$ is an injective map of $\infty$-compositions. Choose $i \in \im(f)$ such that $\lambda_i=\infty$. Define $g \colon \langle \lambda \rangle \to \langle \mu \rangle$ by
\begin{displaymath}
g(j) = \begin{cases}
f^{-1}(j) & \text{if $j \in \im(f)$} \\
f^{-1}(i) & \text{if $j \not\in \im(f)$.} \end{cases}
\end{displaymath}
Note that $g \circ f = \id_{[n]}$. Let $h \colon \langle \lambda \rangle \to \langle \lambda \rangle$ be the composition $f \circ g$; explicitly,
\begin{displaymath}
h(j) = \begin{cases}
j & \text{if $j \in \im(f)$} \\
i & \text{if $j \not\in \im(f)$.} \end{cases}
\end{displaymath}
Note that $h$ defines a map of $\infty$-compositions $\lambda \to \lambda$.

We claim that $\alpha_f(Z)=\alpha_g^{-1}(Z)$. Indeed, suppose $x \in \alpha_f(Z)$. Then $x=\alpha_f(y)$ for some $y \in Z$. We thus have $\alpha_g(x)=\alpha_g(\alpha_f(y))=\alpha_h(y) \in Z$ since $Z$ is stable under $\alpha_h$. Thus $x \in \alpha_g^{-1}(Z)$. Conversely, suppose $x \in \alpha_g^{-1}(Z)$. Since $g \circ f=\id_{[n]}$, it follows that $\alpha_f \circ \alpha_g$ is the identity. Thus $x=\alpha_f(\alpha_g(z)) \in \alpha_f(Z)$, as $\alpha_g(z) \in Z$. This proves the claim. It follows that $\alpha_f(Z)$ is closed.

Now suppose that $f$ is an arbitrary correspondence. Let $f=(f_1 \colon \rho \psurj \mu, f_2 \colon \rho \to \lambda)$ where $f_1$ is a principal projection. Factor $f_2$ as $g \circ h$ where $h$ is a principal projection and $g$ is an injection. Thus $\alpha_f(Z)=\alpha_{f_1}^{-1}(\alpha_h(\alpha_g(Z)))$. By the first paragraph, $\alpha_g(Z)$ is closed. Since $\alpha_h$ is a closed immersion, it follows that $\alpha_h(\alpha_g(Z))$ is closed. Thus $\alpha_f(Z)$ is closed.
\end{proof}

\subsection{Construction of C3 subvarieties}

For $Z \subset \cX_{\lambda}$ define $\Gamma_{\lambda}(Z)=\Gamma_{\lambda}^{\circ}(Z^e)$.  The following theorem describes the most important properties of this construction.

\begin{theorem}
\label{thm:gamma}
Let $Z \subset \cX_{\lambda}$ be an arbitrary subset. Then $\cZ=\Gamma_{\lambda}(Z)$ is a C3 subvariety of $\cX$. Furthermore, if $\cZ'$ is any C3 subvariety such that $\cZ'_{\lambda}$ contains $Z$ then $\cZ'$ contains $\cZ$.
\end{theorem}

\begin{proof}
By Theorem~\ref{thm:gamma0}, we know that $\cZ$ is a C2 subset of $\cX$. It thus suffices to show that $\cZ_{\mu}$ is Zariski closed for all $\mu$. By Proposition~\ref{prop:finite-union}, we have
\begin{displaymath}
\cZ_{\mu}
= \Gamma^{\circ}_{\lambda}(Z^e)_{\mu}
= \bigcup_{\substack{f \colon \mu \dashrightarrow \lambda \\ \text{$f$ is good}}} \alpha_f(Z^e).
\end{displaymath}
Since each $\alpha_f(Z^e)$ is Zariski closed (Proposition~\ref{prop:corr-preserves-closure}) and the union is finite, we see that $\cZ_{\mu}$ is Zariski closed. Thus $\cZ$ is a C3 subvariety.

Now, suppose $\cZ'$ is a C3 subvariety such that $\cZ'_{\lambda}$ contains $Z$. Since $\cZ'$ is closed under correspondences (Proposition~\ref{prop:corr-closed}), we see that $\cZ'_{\lambda}$ contains $\bigcup_{f \in \End(\lambda)} \alpha_f(Z)$. Since $\cZ'_{\lambda}$ is also Zariski closed, it follows that it contains $Z^e$. Since $\cZ'$ is a C2 subset such that $\cZ'_{\lambda}$ contains $Z^e$, Theorem~\ref{thm:gamma0} shows that $\cZ'$ contains $\Gamma^{\circ}_{\lambda}(Z^e)=\cZ$. This completes the proof.
\end{proof}

We will require some additional properties of the $\Gamma_{\lambda}$ construction:

\begin{proposition} \label{prop:gamma-aux}
Let $Z \subset \cX_{[\lambda]}$ and put $\cZ=\Gamma_{\lambda}(Z)$. Then $\cZ_{[\lambda]}$ is the Zariski closure of $\bigcup_{f \in \Aut(\lambda)} \alpha_f(Z)$ in $\cX_{[\lambda]}$.
\end{proposition}

\begin{proof}
Let $E=\cX_{\lambda} \setminus \cX_{[\lambda]}$; this is a Zariski closure subset of $\cX_{\lambda}$. Also, let $Z'$ be the Zariski closure of $\bigcup_{f \in \Aut(\lambda)} \alpha_f(Z)$. If $f \colon \lambda \to \lambda$ is any endomorphism that is not an automorphism then $\alpha_f(\cX_{\lambda}) \subset E$. It follows that $Z' \subset Z^e \subset Z' \cup E$, and so $Z^e \cap \cX_{[\lambda]}=Z'$. Since $\cZ=\Gamma_{\lambda}^{\circ}(Z^e)$, the result now follows from Proposition~\ref{prop:gamma-aux3}.
\end{proof}

\begin{proposition} \label{prop:gamma-aux2}
Let $Z$ be an $\Aut(\lambda)$-irreducible closed subvariety of $\cX_{[\lambda]}$. Then $\cZ=\Gamma_{\lambda}(Z)$ is an irreducible C3-subvariety of $\cX$.
\end{proposition}

\begin{proof}
Suppose $\cZ$ can be written as a union $\cZ^1 \cup \cZ^2$ of C3 subvarieties of $\cX$.  By Proposition~\ref{prop:gamma-aux}, we have $\cZ_{[\lambda]} = Z$; on the other hand, $\cZ_{[\lambda]}=\cZ^1_{[\lambda]} \cup \cZ^2_{[\lambda]}$. Since $Z$ is $\Aut(\lambda)$-irreducible, we have $\cZ^1=Z$ or $\cZ^2=Z$, say the former. Since $\cZ^1$ is a C3 subvariety containing $Z$ it must contain $\cZ$ by Theorem~\ref{thm:gamma}, and so $\cZ=\cZ^1$, which completes the proof.
\end{proof}

\begin{remark}
We note that if $Z$ is $\End(\lambda)$-stable closed subvariety of $\cX_{\lambda}$, then $\Gamma_{\lambda}(Z)_{\lambda}$ may strictly contain $Z$. To see this, let $W = \Spec(\bC)$,  $\lambda = (\infty, 2, 1, 1)$, and let $Z$ be the $\End(\lambda)$-closure of $\{ (1,2,3,3) \}$. Then it is easy to see that $(1, 3, 2,2)$ is not in $Z$. On the other hand, let $f = (f_1, f_2)$ be the correspondence where $f_1 \colon (\infty, 1, 1, 1, 1) \to \lambda$ is the principal surjection which combines the second and the third coordinates and $f_2 \colon (\infty, 1, 1, 1, 1) \to \lambda$ is the principal surjection which combines the third and the fourth coordinates. Now it is easy to see that $(1, 3, 2,2) \in \alpha_f(\{ (1,2,3,3) \})$. Thus $\Gamma_{\lambda}(Z)_{\lambda}$ strictly contain $Z$.
\end{remark}


\section{Classification of subvarieties of $\cX$}

\textit{We assume $A$ is noetherian throughout this section}

\subsection{Bounded subvarieties of $\cX$}

Let $\cZ$ be a subset of $\cX$. We say that $\cZ$ is {\bf bounded} if there is some $n$ such that $\ell(\lambda) \le n$ for all $\lambda$ with $\cZ_{[\lambda]}$ non-empty.  For an $\infty$-composition $\lambda$, let $\cX^{\le n}_{\lambda} \subset \cX_{\lambda}$ be the subset where the coordinates take at most $n$ distinct values. This is a Zariski closed subset, as it is defined by the $n$-variable discriminants. One easily sees that $\cX^{\le n}$ is a C3-subvariety of $\cX$. Essentially by definition, a subset of $\cX$ is bounded if and only if it is contained in $\cX^{\le n}$ for some $n$.

The purpose of this section is to establish the following result:

\begin{proposition} \label{prop:bd}
Let $\cZ$ be a C3 subvariety of $\cX$. Then we have a decomposition $\cZ=\cZ_1 \cup \cZ_2$ where $\cZ_1=\cX_{W'}$ for some closed subset $W'$ of $W$ and $\cZ_2$ is a bounded C3 subvariety.
\end{proposition}

Here $\cX_{W'}$ denotes the subset of $\cX$ given by $\cX_{W',\lambda}=\bA_{W'}^{\langle \lambda \rangle}$. We require several lemmas. We say that a $K$-point $w \in W(K)$ is {\bf ever-present} in $\cZ$ if $\cZ$ contains $\cX_{\{w\}}$. This notion is invariant under field extension, and thus we can speak of ever-present points of $\cZ$.

\begin{lemma} \label{lem:bd1}
	Let $K$ be a field, and let $w$ be a $K$-point of $W$. Let $\cZ$ be a C3 subvariety of $\cX$. Suppose that there is an element $a \in K$ such that $\cZ_{(\infty,1^n)}(K)$ contains the $K$-points in $\{(w,a)\} \times \bA^n$ for all $n$. Then $w$ is ever-present in $\cZ$. 
\end{lemma}
\begin{proof}
	Suppose $\lambda=(\lambda_1,\lambda_2,\ldots,\lambda_r)$ is an $\infty$-partition with $\lambda_1=\infty$ and $\lambda_i<\infty$ for $i>1$. We can then find a principal surjection $(\infty,1^n) \to \lambda$ with $n=\lambda_2+\cdots+\lambda_r$. Since $\cZ$ is C1, we have $\cZ_{\lambda}=\cZ_{(\infty,1^n)} \cap \cX_{\lambda}$, and so $\cZ_{\lambda}(K)$ contains the $K$-points in $\{(w,a)\} \times \bA^{r-1}$.
	
	Now suppose that $\lambda=(\lambda_1, \ldots, \lambda_r)$ is an $\infty$-partition with $\lambda_1=\cdots=\lambda_s=\infty$ and $\lambda_{s+1}<\infty$. Let $\lambda^n=(\infty, n, \ldots, n, \lambda_{s+1}, \ldots, \lambda_r)$, so that $\lambda$ is the limit of the $\lambda^n$'s. Since $\cZ$ is C2, we have $\cZ_{\lambda}=\bigcap_{n \ge 1} \cZ_{\lambda^n}$ (Proposition~\ref{prop:continuity}). Since each $\cZ_{\lambda^n}(K)$ contains the $K$-points in $\{(w,a)\} \times \bA^{r-1}$, so does $\cZ_{\lambda}$.
	
	Keeping $\lambda$ as above, let $\mu=(\infty, \lambda_1, \ldots, \lambda_r)$. We have a natural injection $\lambda \to \mu$. By Proposition~\ref{prop:C2-map}, we see that $\cZ_{\lambda}$ contains the image of $\cZ_{\mu}$ under the projection $\cX_{\mu} \to \cX_{\lambda}$. However, any $K$-point of the form $(w, b_1, \ldots, b_r) \in \cX_{\lambda}(K)$ is the image of the point $(w, a,b_1,\ldots,b_r) \in \cX_{\mu}(K)$, which belongs to $\cZ_{\mu}(K)$. We thus see that $\cZ_{\lambda}(K)$ contains the $K$-points in $\{w\} \times \bA^{\ell(\lambda)}$, as claimed.
\end{proof}

\begin{lemma} \label{lem:bd2}
Let $\cZ$ be a C3 subvariety of $\cX$ with no ever-present point. Then $\cZ$ is bounded.
\end{lemma}

\begin{proof}
We prove the contrapositive. Thus suppose that $\cZ$ is unbounded, and we will produce an ever-present point. It suffices to treat the case where $W$ is affine, so say $W=\Spec(A)$.

We first claim that $\cZ_{[\infty,1^n]}$ is non-empty for all $n \ge 0$. Indeed, let $n$ be given. Since $\cZ$ is unbounded, there is an $\infty$-composition $\lambda$ with $\ell(\lambda) \ge n$ such that $\cZ_{[\lambda]}$ is non-empty. Pick an injection $f \colon (\infty,1^n) \to \lambda$, which is possible. Then $\alpha_f$ maps $\cZ_{[\lambda]}$ into $\cZ_{[\infty,1^n]}$, and so the latter is non-empty, as claimed.

Put $W'=W \times \bA^1$. Let $I$ be a finite subset of $[\infty]$, and let $\lambda_I$ be the $\infty$-composition on $I \cup \{\ast\}$ that is~1 on $I$ and $\infty$ on $\ast$. Put $\cX_I=\cX_{\lambda_I}$, which we identify with $W' \times \bA^I$, and let $\cX'_I$ be the open subvariety $W' \times \bA^{[I]}$. Put $\cZ_I=\cZ_{\lambda_I}$ and $\cZ'_I=\cZ_I \cap \cX'_I$. Let $R_I=A[u,\xi_i]_{i \in I}$ and $S_I=R_I[(\xi_i-\xi_j)^{-1}]$. We identify $A[u]$, $R_I$, and $S_I$ with the coordinate rings of $W'$, $\cX_I$, and $\cX'_I$. Let $\fa_I \subset S_I$ be the ideal corresponding to $\cZ'_I$. Since $\cZ'_I$ is non-empty, the ideal $\fa_I$ is not the unit ideal.

Suppose that $J$ is a second finite subset of $[\infty]$ containing $I$. We have a natural injection of $\infty$-compositions $\lambda_I \to \lambda_J$ that induces a projection map $\cX_J \to \cX_I$ which carries $\cX'_J$ into $\cX'_I$, and hence $\cZ'_J$ into $\cZ'_I$. The map $\cX_J \to \cX_I$ corresponds to the canonical inclusion of rings $R_I \to R_J$. We thus see that, under this inclusion, $\fa_I$ is contained in $\fa_J$.

Let $R=A[u,\xi_i]_{i \in [\infty]}$ and let $S=R[(\xi_i-\xi_j)^{-1}]_{i \neq j}$. We identify $R$ and $S$ with the direct limits of $R_I$ and $S_I$, taken over all finite subsets $I$ of $[\infty]$. Let $\fa \subset S$ be the direct limit of the $\fa_I$'s. This is not the unit ideal: indeed, if $\fa$ contained~1 then some $\fa_I$ would contain~1, which it does not. By Proposition~\ref{prop:extension-contraction}, we see that $\fa$ is the extension of the ideal $\fb=\fa^c$ of $A[u]$.

Let $(w,a)$ be a $K$-point of $V(\fb) \subset W \times \bA^1$. Then $(w,a)$ determines a map $A[u] \to K$ whose kernel contains $\fb$.  Let $b=(b_1,b_2,\ldots)$ be any sequence of distinct elements of $K$. We then have a homomorphism $\phi_b \colon S \to \bk$ given by mapping an element of $A[u]$ to its image in $K$ and $\xi_i$ to $b_i$, and the kernel of $\phi_b$ contains $\fa$. It follows that $\ker(\phi_b \vert_{S_I})$ contains $\fa_I$, and so $(a,b_1,\ldots,b_n)$ is a $K$-point of $\cZ'_I$. Since this holds for all choices of distinct $b_1, \ldots, b_n$, we see that the set of $K$-points in $\cZ_I$ contains $\{(w,a)\} \times \bA^{[I]}$, and thus contains its closure $\{(w,a)\} \times \bA^I$. By Lemma~\ref{lem:bd1}, we see that  $w$ is an ever-present  $K$-point in $\cZ$. Thus $\cZ$ has an ever-present point.
\end{proof}

\begin{lemma} \label{lem:bd3}
Let $Z$ be a closed subvariety of $W \times \bA^n$. Let $W'$ be the set of points $w \in W$ such that $Z$ contains $\{w\} \times \bA^n$. Then $W'$ is closed.
\end{lemma}

\begin{proof}
It suffices to treat the case where $W$ is affine, so say $W \subset \bA^m$. Thus $Z$ is a closed subvariety of $\bA^m \times \bA^n$. Let $f_i(\xi_1, \ldots, \xi_m, \eta_1, \ldots, \eta_n)=0$, for $1 \le i \le r$, be equations definining $Z$. If $x=(x_1, \ldots, x_m)$ is a point of $\bA^m$, then $Z$ contains $\{x\} \times \bA^n$ if and only if the polynomials $f_i(x_1, \ldots, x_m, \eta_1, \ldots, \eta_n)$ vanish identically, i.e., all coefficients vanish. Thus this locus is defined by a system of polynomial equations, and thus closed.
\end{proof}

\begin{lemma} \label{lem:bd4}
Let $\cZ$ be a C3 subvariety of $\cX$ and let $W' \subset W$ be the set of ever-present points. Then $W'$ is closed.
\end{lemma}

\begin{proof}
For an $\infty$-composition $\lambda$, let $W_{\lambda} \subset W$ be the set of points $w$ such that $\cZ_{\lambda}$ contains $\cX_{\{w\}, \lambda}$. By Lemma~\ref{lem:bd3}, we see that $W_{\lambda}$ is closed. Since $W'=\bigcap W_{\lambda}$, with the intersection over all $\infty$-compositions $\lambda$, we see that $W'$ is also closed.
\end{proof}

\begin{proof}[Proof of Proposition~\ref{prop:bd}]
Let $\cZ \subset \cX$ be a given C3-subvariety. Let $W' \subset W$ be the set of ever-present points, which is closed by Lemma~\ref{lem:bd4}, and put $\cZ_1=\cX_{W'} \subset \cZ$. Let $U=W \setminus W'$, and write $U=U_1 \cup \cdots \cup U_r$ with each $U_i$ affine, which is possible since $W$ is noetherian. Let $\cZ'_i=\cZ \cap \cX_{U_i}$. It is clear that $\cZ'_i$ is a C3-subvariety of $\cX_{U_i}$ that contains no ever-present points, and is therefore bounded by Lemma~\ref{lem:bd2}. Thus there is some $n$ such that $\cZ'_i$ is contained in $\cX^{\le n}_{U_i}$ for all $1 \le i \le r$. Let $\cZ_2=\cZ \cap \cX^{\le n}$, which is a bounded C3-subvariety of $\cZ$. Since $\cZ_2$ contains $\cZ'_i$ for all $i$, we have $\cZ=\cZ_1 \cup \cZ_2$, as required.
\end{proof}

\subsection{The noetherian property for $\cX$}

The purpose of this section is to prove that $\cX$ is noetherian, in the following sense:

\begin{proposition} \label{prop:noeth-C3}
The descending chain condition holds for C3 subvarieties of $\cX$.
\end{proposition}

%

\begin{proof}
Recall that an ideal $I$ of a poset $(P, \preceq)$ is a subset such that $x \in I, x \preceq y \implies y \in I$. Let $\Lambda$ be the poset $\infty$-partitions with respect to $\preceq$. Being a subposet of $\wt{\Lambda}$,  we see that $\Lambda$ is a wqo (Proposition~\ref{prop:min-are-finite}). It follows that the poset of ideals of $\Lambda$ under inclusion satisfies ACC; see \cite[\S 2]{catgb}.

We first deal with the bounded case. Let	$\cZ^1 \supset \cZ^2 \supset \cdots$ be a decreasing chain of bounded C3-subvarieties of $\cX$. Let $\Lambda_n$ be the set given by
\begin{displaymath}
\Lambda_n = \{ \lambda \in \Lambda \colon \forall \mu \succeq \lambda, \cZ^k_{[\mu]} \text{ stabilizes for } k \ge n\}.
\end{displaymath}
Then $\Lambda_1 \subset \Lambda_2 \subset \cdots$ is an increasing chain of ideals of $\Lambda$, and so it must stabilize. Let $N$ be large enough such that $\Lambda_n = \Lambda_N$ for all $n \ge N$. Suppose, if possible, $\Lambda \setminus \Lambda_N \neq \emptyset$. We claim that $\Lambda \setminus \Lambda_N$ has a maximal element:

To see this let $\wt{\Lambda}_n$ be the set given by
\begin{displaymath}
\wt{\Lambda}_n = \{ \lambda \in \Lambda \colon \cZ^k_{[\lambda]} \text{ stabilizes for } k \ge n\}.
\end{displaymath}
Note that $\wt{\Lambda}_n$ may not be an ideal of $\Lambda$. Moreover, the set $\Lambda \setminus \wt{\Lambda}_N$ is nonempty as $\Lambda \setminus \Lambda_N$ is nonempty. We first show that $\Lambda \setminus \wt{\Lambda}_N$ has a maximal element. By Zorn's lemma, it suffices to show that every increasing chain $\lambda^1 \preceq \lambda^2 \preceq \cdots$ in $\Lambda \setminus \wt{\Lambda}_N$ has an upper bound in $\Lambda \setminus \wt{\Lambda}_N$. Since $\cZ^1$ is bounded, there exists a $d$ such that $\ell(\lambda) >d$ implies $\lambda \in \Lambda^1 \subset \Lambda_N \subset \wt{\Lambda}_N$. This shows that $\ell(\lambda^i) \le d$ for all $i$. Suppose that the limiting value of $\ell(\lambda^i)$ is some $d' \le d$. We may as well assume that $\ell(\lambda^i) = d'$ for all $i$. Then we can regard each $\lambda^i$ as an $\infty$-weighting on the set $[d']$, and so we have $\lambda^1 \le \lambda^2 \le \cdots$. Let $\lambda \in \Lambda$ be the limit of these $\infty$-compositions. By noetherianity, $\cZ^N_{\lambda^1} \supset \cZ^N_{\lambda^2} \supset \cdots$ stabilizes for each $k$. So by the continuity of $\cZ^N$, we see that $\cZ^N_{[\lambda]} = \cZ^N_{[\lambda^i]}$ for $i \gg 0$. If $\lambda$ were in $\wt{\Lambda}_N$, then for such an $i$ and all $n \ge N$ we would have
\begin{displaymath}
\cZ^n_{[\lambda^i]} \supset \cZ^n_{[\lambda]} = \cZ^N_{[\lambda]} =   \cZ^N_{[\lambda^i]} \supset  \cZ^n_{[\lambda^i]},
\end{displaymath}
where the first equality comes from $\lambda \in \wt{\Lambda}_N$, and the second equality is true as $i \gg 0$. This shows that $\cZ^n_{\lambda^i} = \cZ^N_{\lambda^i}$ for all $n \ge N$ (by Proposition~\ref{prop:C1-contain}), which contradicts $\lambda^i \in \Lambda \setminus \wt{\Lambda}_N$. Thus $\lambda \in \Lambda \setminus \wt{\Lambda}_N$. By Zorn's lemma, there exists a maximal element, say $\lambda$, in  $\Lambda \setminus \wt{\Lambda}_N$. It is clear that $\lambda$ is also a maximal element in $\Lambda \setminus \Lambda_N$. This proves the claim.

Now let $\lambda$ be a maximal element in $\Lambda \setminus \Lambda_N$. By noetherianity, there exists an $N'$ such that $\cZ^1_{\lambda} \supset \cZ^2_{\lambda} \supset \cdots$ stabilizes for $n \ge N'$. Since $\lambda \notin \Lambda_N$, we must have $N' > N$. But then $\lambda' \in \Lambda_{N'} \setminus \Lambda_{N}$, contradicting the definition of $N$. Thus $\Lambda = \Lambda_N$. By the previous lemma, the increasing chain $\cZ^1 \supset \cZ^2 \supset \cdots$ stabilizes for $n \ge N$. This completes the proof in the bounded case.

In general, let $\cZ^1 \supset \cZ^2 \supset \cdots$ be a decreasing chain of (non necessarily bounded) C3-subvarieties of $\cX = \cX_W$. Let $\cX_{W_i}$ be the ever-present part of $\cZ^i$ as in Proposition~\ref{prop:bd}. Then $W_1 \supset W_2 \supset \cdots$ is a decreasing chain of closed subvarities of $W$. By noetherianity of $W$, this decreasing chain stabilizes. So we may as well assume that $W_i = W'$ for all $i \ge 1$. Set $U = W \setminus W'$, and let $\cY^i = \cZ^i \cap \cX_U$. By Lemma~\ref{lem:bd2}, $\cY^i$ is bounded for each $i$. By the bounded case, the decreasing chain $\cY^1 \supset \cY^2 \supset \cdots$ stabilizes. Since $\cZ^i = \cX_{W'} \cup \cY^i$, the decreasing  chain $\cZ^1 \supset \cZ^2 \supset \cdots$ also stabilizes, completing the proof.
\end{proof}

The following corollary follows immediately from the proposition:

\begin{corollary}
Every C3 subvariety is a finite union of irreducible C3 subvarieties.
\end{corollary}

\subsection{The classification theorem}

The following is the classification theorem for irreducible C3 subvarieties of $\cX$.

\begin{theorem}
\label{thm:classification-cZ}
Let $\cZ$ be an irreducible C3 subvariety of $\cX$.
\begin{enumerate}
\item If $\cZ$ is bounded then $\cZ=\Gamma_{\lambda}(Z)$ for a unique $\infty$-partition $\lambda$ and a unique $\Aut(\lambda)$-irreducible subvariety $Z$ of $\cX_{[\lambda]}$.
\item If $\cZ$ is  unbounded, then $\cZ = \cX_{W'}$ for a unique irreducible closed subset $W' \subset W$. 
\end{enumerate}
\end{theorem}

\begin{proof}
Part (b) follows immediately from Proposition~\ref{prop:bd}. For Part (a), let $\Lambda_{\cZ}$ be the set of partitions $\lambda$ such that $\cZ_{[\lambda]}$ is nonempty. Since $\cZ$ is bounded, there exists an $N$ such that the number of parts in any $\lambda \in \Lambda_{\cZ}$ is bounded by $N$. It follows that if $\lambda^1 \preceq \lambda^2 \preceq \cdots$ is a chain in $\Lambda_{\cZ}$, then the number of parts in $\lambda^n$ is eventually independent of $n$. Ignoring the first few entries, we can think of each of $\lambda^n$ to be an $\infty$-composition on the same underlying set. By Proposition~\ref{prop:continuity}, we see that the limit $\lambda$ of $\lambda^n$ is also in $\Lambda_{\cZ}$. By Zorn's lemma, every element in $\Lambda_{\cZ}$ is contained in a maximal element. By Proposition~\ref{prop:min-are-finite}, there are no infinite antichains in  $(\Lambda_{\cZ}, \preceq)$. So there are finitely many maximal elements in $\Lambda_{\cZ}$, call them $\mu^1, \ldots, \mu^n$. Let $\cZ^i$ be the minimal C3-subvariety of $\cX$ such that $\cZ^i_{\mu^i} = \cX_{\mu^i}$. By Theorem~\ref{thm:gamma}, we see that $\cZ^i_{\lambda} = \cX_{\lambda}$ for each $\lambda \preceq \mu_i$. So $\cZ \subset \bigcup_{1 \le i \le n} \cZ^i$. By irreducibility, there exists an $i$ such that $\cZ \subset \cZ^i$. Since $\cZ^i_{[\lambda]}$ is nonempty if and only if $\lambda \preceq \mu^i$, it follows that $\mu^i$ is the unique maximal element of $\Lambda_{\cZ}$. Set $\mu\coloneq  \mu_i, Z \coloneq \cZ_{[\mu]}$, and let $\cZ'$ be the minimal C3-subvariety of $\cX$ such that $\cZ'_{\mu} \supset Z$ (which is as described in the Theorem~\ref{thm:gamma}).  By minimality, we have $\cZ' \subset \cZ$. We claim that $\cZ' = \cZ$. 
	
To see this, let $\Lambda' = \{\lambda \in \Lambda_{\cZ} \colon \cZ'_{[\lambda]} = \cZ_{[\lambda]} \}$. Let $\lambda^1 \preceq \lambda^2 \preceq \cdots$ be a chain in $\Lambda_{\cZ} \setminus \Lambda'$. By boundedness of $\cZ$, the number of parts in $\lambda^n$ is eventually constant in $n$. Ignoring the first few entries, we can think of each of $\lambda^n$ to be an $\infty$-composition on the same underlying set. By noetherianity (Proposition~\ref{prop:noeth-C3}), we see that the chains $\cZ_{\lambda^1} \supset \cZ_{\lambda^2} \supset \cdots$ and $\cZ'_{\lambda^1} \supset \cZ'_{\lambda^2} \supset \cdots$ stabilize. So Proposition~\ref{prop:continuity} applied to $\cZ$ and $\cZ'$, we see that for large enough $N$, we have $\cZ_{[\lambda^N]} = \cZ_{[\lambda]}$ and $\cZ'_{[\lambda^N]} = \cZ'_{[\lambda]}$. It follows that $\cZ'_{[\lambda]} = \cZ'_{[\lambda^N]} \subsetneq \cZ_{[\lambda^N]} = \cZ_{[\lambda]}$, and so $\lambda \in \Lambda_{\cZ} \setminus \Lambda'$. By Zorn's lemma, every element in $\Lambda_{\cZ} \setminus \Lambda'$ is contained in a maximal element. By Proposition~\ref{prop:min-are-finite}, there are no infinite antichains in  $(\Lambda, \preceq)$. So there are finitely many maximal elements in $\Lambda_{\cZ} \setminus \Lambda'$, call them $\mu^1, \ldots, \mu^n$. As in the previous paragraph, let $\cZ^i$ be the minimal C3-subvariety of $\cX$ such that $\cZ^i_{\mu^i} = \cX_{\mu^i}$. Let $\cZ'' = \cZ \subset \cZ' \cup \bigcup_{i=1}^n \cZ^i$. Our construction implies that $\cZ_{[\lambda]} \subset \cZ''_{[\lambda]}$ for each $\lambda$. By Proposition~\ref{prop:C1-contain}, we see that $\cZ \subset \cZ'' =  \cZ' \cup \bigcup_{i=1}^n \cZ^i$. Since $\cZ^i_{[\mu]} =\emptyset$, the irreducibility of $\cZ$ implies that $\cZ \subset \cZ'$. This establishes the claim.
	
Finally, suppose $Z = \cZ_{[\mu]}$ can be written as $Z = \bigcup_{i=1}^n Z_n$ where each $Z_n$ is $\Aut(\mu)$-stable Zariski closed $\Aut(\mu)$-irreducible subset of $\cX_{[\mu]}$. Let $\cZ^i$ be the irreducible C3-subvariety corresponding to $Z_i$. Then by Theorem~\ref{thm:gamma}, we see that $\cZ = \bigcup_{i=1}^n \cZ^i$. By irreducibility of $\cZ$ again, we see that $\cZ = \cZ^i$ for some $i$. This finishes the proof.
\end{proof}

\section{Extending from $\cX$ to $\cX^+$}

In the next section, we study the correspondence between subvarieties of $\fX$ and subvarieties of $\cX$. To carry this out, we extend $\cX$ to also include finite compositions. We now study this construction.

Given a non-empty generalized composition $\lambda$, we let $\cX^+_{\lambda}=\bA^{[\lambda]}_W$. If $\lambda$ is infinite then $\cX^+_{\lambda}$ is simply $\cX_{\lambda}$. We now consider subsets of $\cX^+$. We define the conditions C1, C2, and C3 exactly as for $\cX$. The elementary results about subsets of $\cX$ extend to subsets of $\cX^+$ without difficulty. Given a subset $\cZ$ of $\cX^+$, we let $\cZ \cap \cX$ be the subset of $\cX$ it induces, i.e., $(\cZ \cap \cX)_{\lambda}=\cZ_{\lambda}$ for $\infty$-compositions $\lambda$. We now define a construction in the opposite direction.

Let $\cZ$ be a subset of $\cX$. Define a subset $\cZ^+$ of $\cX^+$ as follows. For an $\infty$-composition $\lambda$ we put $\cZ^+_{\lambda}=\cZ_{\lambda}$. Now suppose $\lambda$ is a finite non-empty composition. Let $\lambda^+=\lambda \cup \{\infty\}$, and let $f \colon \lambda \to \lambda^+$ be the standard inclusion. We define $\cZ^+_{\lambda}=\alpha_f(\cZ_{\lambda^+})$. Note that $\cZ^+$ extends $\cZ$, that is, $\cZ^+ \cap \cX=\cZ$. We now study this construction in more detail.

\begin{lemma} \label{lem:plus-finite}
Let $\cZ$ be a C2 subset of $\cX$ and let $f \colon \lambda \to \mu$ be a map of finite partitions. Then $\alpha_f(\cZ^+_{\mu}) \subset \cZ^+_{\lambda}$. If $f$ is a principal surjection then $\cZ^+_{\mu}=\alpha_f^{-1}(\cZ^+_{\lambda})$.
\end{lemma}

\begin{proof}
Let $f^+ \colon \lambda^+ \to \mu^+$ be the induced map. Note that $\cX_{\lambda^+}=\cX_{\lambda} \times \bA^1$ and similarly for $\cX_{\mu^+}$, and $\alpha_{f^+}=\alpha_f \times \id$. We have a commutative diagram
\begin{displaymath}
\xymatrix@C=4em{
\cX_{\mu} \times \bA^1 \ar[r]^{\alpha_f \times \id} \ar[d]_q & \cX_{\lambda} \times \bA^1 \ar[d]^p \\
\cX_{\mu} \ar[r]^{\alpha_f} & \cX_{\lambda} }
\end{displaymath}
where $p$ and $q$ are the projection maps. Now, suppose $x \in \cZ^+_{\mu}$. Since $\cZ^+_{\mu}=q(\cZ_{\mu^+})$, it follows that there is some $y \in \bA^1$ such that $(x,y) \in \cZ_{\mu^+}$. By Proposition~\ref{prop:C2-map}, we see that $\alpha_{f^+}(x,y)=(\alpha_f(x), y)$ belongs to $\cZ_{\lambda^+}$. Thus $p(\alpha_f(x),y)=\alpha_f(x)$ belongs to $\cZ^+_{\lambda}=p(\cZ_{\lambda^+})$.

Now suppose that $f$ is a principal surjection. Then $f^+$ is as well. Suppose $x \in \cX^+_{\mu}$ and $\alpha_f(x) \in \cZ^+_{\lambda}$. Then there is some $y \in \bA^1$ such that $(\alpha_f(x),y) \in \cZ_{\lambda^+}$. Thus $\alpha_{f^+}(x,y) \in \cZ_{\lambda^+}$ and so $(x,y) \in \cZ_{\mu^+}$ since $\cZ$ is C1. Hence $x=q(x,y)$ belongs to $\cZ^+_{\mu}$. This completes the proof.
\end{proof}

\begin{proposition}
Let $\cZ$ be a C2 subset of $\cX$. Then $\cZ^+$ is a C2 subset of $\cX^+$. Moreover, if $\cZ'$ is any C2 subset of $\cX^+$ such that $\cZ' \cap \cX$ contains $\cZ$ then $\cZ'$ contains $\cZ^+$.
\end{proposition}

\begin{proof}
We first note that $\cZ$ is C1. Indeed, suppose that $f \colon \lambda \to \mu$ is a principal surjection. If $\lambda$ is infinite, so is $\mu$, and then $\alpha_f^{-1}(\cZ^+_{\lambda})=\cZ^+_{\mu}$ since $\cZ$ is C1. If $\lambda$ is finite, so is $\mu$, and then $\alpha_f^{-1}(\cZ^+_{\lambda})=\cZ^+_{\mu}$ by Lemma~\ref{lem:plus-finite}.

Now suppose that $f \colon \lambda \to \mu$ is an injection of generalized partitions. We claim that $\alpha_f(\cZ^+_{\mu}) \subset \cZ^+_{\lambda}$. If $\lambda$ is infinite then so is $\mu$, and this follows from Proposition~\ref{prop:C2-map}. Thus suppose $\lambda$ is finite. If $\mu$ is finite then the claim follows from Lemma~\ref{lem:plus-finite}. Thus suppose that $\mu$ is infinite. We can then factor $f$ as $g \circ h$ where $h \colon \lambda \to \lambda^+$ is the standard inclusion and $g \colon \lambda^+ \to \mu$ is any extension of $f$ (simply map $\infty \in \lambda^+$ to any element of $\mu^{-1}(\infty)$). We have
\begin{displaymath}
\alpha_f(\cZ^+_{\mu})=\alpha_h(\alpha_g(\cZ_{\mu})) \subset \alpha_h(\cZ_{\lambda^+}) = \cZ^+_{\lambda}
\end{displaymath}
where in the second step we used Proposition~\ref{prop:C2-map}. This establishes the claim. Combined with the first paragraph, we see that $\cZ^+$ is closed under correspondences.


Now suppose $x \in \cX^+_{\lambda}$ is approximable by $\cZ^+$. We show $x \in \cZ^+_{\lambda}$. First suppose that $\lambda$ is finite. Let $N$ be larger than all parts of $\lambda$. Let $\mu$ be a generalized composition with $\langle \lambda \rangle \subset \langle \mu \rangle$ and $\mu_i \ge \min(N,\lambda_i)$ for all $i \in \langle \lambda \rangle$ such that $x \in \alpha_f(\cZ^+_{\mu})$ where $f \colon \langle \lambda \rangle \to \langle \mu \rangle$ is the inclusion. By our choice of $N$, we have $\mu_i \ge \lambda_i$ for all $i$, and so $f \colon \lambda \to \mu$ is a map of generalizaed partitions. Thus $\alpha_f$ carries $\cZ^+_{\mu}$ into $\cZ^+_{\lambda}$, and so $x \in \cZ^+_{\lambda}$, as required.

Now suppose $\lambda$ is infinite. We claim that $x$ is in fact approximable by $\cZ$. Let $N$ be given. Since $x$ is $N$-approximable by $\cZ^+$ we can find a generalized composition $\mu$ and an inclusion $f \colon \langle \lambda \rangle \to \langle \mu \rangle$ such that $\mu_{f(i)} \ge \min(N, \lambda_i)$ for all $i \in \langle \lambda \rangle$ and $x \in \alpha_f(\cZ_{\mu})$. If $\mu$ is infinite then this shows that $x$ is $N$-approximable by $\cZ$. Thus suppose $\mu$ is finite. Let $\nu=\mu \cup \{\infty\}$ and let $g \colon \mu \to \nu$ be the standard inclusion. By definition, $\cZ_{\mu}=\alpha_g(\cZ_{\nu})$, and so $x \in \alpha_{g \circ f}(\cZ_{\nu})$. Since $g \circ f$ is injective and $\nu_{g(f(i))}=\mu_{f(i)} \ge \lambda_i$ for all $i \in \langle \lambda \rangle$, this shows that $x$ is $N$-approximable by $\cZ$. As this holds for all $N$, the claim follows. Since $\cZ$ is C2, we conclude that $x \in \cZ_{\lambda}$, and thus $\cZ^+$ is C2.

Finally, suppose that $\cZ'$ is a C2 subset of $\cX^+$ such that $\cZ' \cap \cX$ contains $\cZ^+$. Let $\lambda$ be a non-empty composition, and let $f \colon \lambda \to \lambda^+$ be the standard inclusion. Then $\cZ'_{\lambda^+}$ contains $\cZ_{\lambda^+}$ by assumption. Thus $\alpha_f(\cZ'_{\lambda^+})$ contains $\alpha_f(\cZ_{\lambda^+})=\cZ^+_{\lambda}$. Since $\cZ'$ is C2, we have $\alpha_f(\cZ'_{\lambda^+}) \subset \cZ'_{\lambda}$. Thus $\cZ'_{\lambda}$ contains $\cZ^+_{\lambda}$. If $\lambda$ is an $\infty$-composition then $\cZ'_{\lambda}$ contains $\cZ^+_{\lambda}=\cZ_{\lambda}$ by assumption. Thus $\cZ'$ contains $\cZ^+$. This completes the proof.
\end{proof}

\begin{proposition} \label{prop:plus-gamma-circ}
Let $Z$ be a subset of $\cX_{\lambda}$ for some $\infty$-composition $\lambda$ and let $\cZ=\Gamma^{\circ}_{\lambda}(Z)$. Then for any non-empty generalized composition $\mu$, we have
\begin{displaymath}
\cZ^+_{\mu}
= \bigcup_{f \colon \mu \dashrightarrow \lambda} \alpha_f(Z).
\end{displaymath}
\end{proposition}

\begin{proof}
If $\mu$ is infinite, this is simply the definition of $\Gamma_{\lambda}^{\circ}$. Suppose $\mu$ is finite. Since $\cZ^+$ is closed under compositions and $Z \subset \cZ^+_{\lambda}$, we have $\alpha_f(Z) \subset \cZ_{\mu}$ for any $f \colon \mu \dashrightarrow \lambda$. Let $g \colon \mu \to \mu^+$ be the standard inclusion. Then
\begin{displaymath}
\cZ^+_{\mu}=\alpha_g(\cZ_{\mu^+}) = \bigcup_{f \colon \mu^+ \dashrightarrow \lambda} \alpha_g(\alpha_f(Z)) \subset \bigcup_{h \colon \mu \dashrightarrow \lambda} \alpha_h(Z).
\end{displaymath}
The first equality is the definition of $\cZ^+_{\mu}$, the second follows from the definition of $\cZ_{\mu^+}$, and the third follows by taking $h$ to be a composition $f \circ g$ and applying Proposition~\ref{prop:comp-corr}. This completes the proof.
\end{proof}

\begin{proposition} \label{prop:corr-preserves-closure-alt}
Let $\lambda$ be an $\infty$-composition and let $Z$ be an $\End(\lambda)$-stable Zariski closed subset of $\cX_{\lambda}$. Let $f \colon \mu \dashrightarrow \lambda$ be a correspondence of generalized compositions. Suppose $\vert \mu \vert$ is larger than the sum of all finite parts of $\lambda$. Then $\alpha_f(Z)$ is Zariski closed in $\cX^+_{\mu}$.
\end{proposition}

\begin{proof}
If $f$ is an injective map then there must be some $i \in \im(f)$ with $\lambda_i=\infty$ by our assumption on $\vert \mu \vert$. With this observation in hand, the proof of Proposition~\ref{prop:corr-preserves-closure} can be carried over verbatim.
\end{proof}

\begin{proposition} \label{prop:C3-plus}
Suppose $A$ is noetherian, and let $\cZ$ be a C3 subvariety of $\cX$. Then $\cZ^+_{\mu}$ is Zariski closed in $\cX^+_{\mu}$ for all generalized compositions $\mu$ with $\vert \mu \vert \gg 0$.
\end{proposition}

\begin{proof}
Since $\cZ \mapsto \cZ^+$ is compatible with finite unions, it suffices to treat the case where $\cZ$ is irreducible. If $\cZ$ is unbounded then $\cZ=\cX_{W'}$ for some $W' \subset W$, and one finds that $\cZ^+=\cX^+_{W'}$. In this case, $\cZ^+$ is a C3 subvariety of $\cX^+$. Now suppose that $\cZ$ is bounded. Then $\cZ=\Gamma_{\lambda}(Z)$ for some $\infty$-composition $\lambda$ and some $\End(\lambda)$-stable closed subvariety $Z$ of $\cX_{\lambda}$. By Proposition~\ref{prop:plus-gamma-circ}, we have
\begin{displaymath}
\cZ^+_{\mu} = \bigcup_{f \colon \mu \dashrightarrow \lambda} \alpha_f(Z).
\end{displaymath}
Since $\mu$ is finite, there are (up to isomorphism) only finitely many principal surjections $\rho \to \lambda$, and thus only finitely many choices for $f$. Thus the above union is finite. If $\vert \mu \vert$ is sufficiently large, then $\alpha_f(Z)$ is Zariski closed by Proposition~\ref{prop:corr-preserves-closure-alt}. The result follows.
\end{proof}

\section{The correspondence between $\fX$ and $\cX$}
\label{sec:correspondence}

\subsection{Setup}

Regard $1^{\infty}$ as a composition on the index set $\langle 1^{\infty} \rangle = [\infty]$ taking the value~1 on each point. Suppose that $\lambda$ is a generalized composition. A {\bf map} $f \colon 1^{\infty} \to \lambda$ is a function $f \colon [\infty] \to \langle \lambda \rangle$ such that $\# f^{-1}(i) \le \lambda_i$ for all $i \in \langle \lambda \rangle$. We say that such a map $f$ is a {\bf principal surjection} if $\lambda_i=\# f^{-1}(i)$ for all $i \in \langle \lambda \rangle$. Given a map $f$, we let $\alpha_f \colon \cX_{\lambda} \to \fX$ be the associated map of schemes. If $f$ is a principal surjection then $\alpha_f$ is a multi-diagonal map, and thus a closed immersion.

Let $\fZ$ be a $\fS^{\rm big}$-subset of $\fX_{\fin}$. We define a subset $\cZ=\Phi(\fZ)$ of $\cX$ as follows. Given an $\infty$-composition $\lambda$, choose a principal surjection $f \colon 1^{\infty} \psurj \lambda$, and put $\cZ_{\lambda}=\alpha_f^{-1}(\fZ)$. Note that if $g \colon 1^{\infty} \psurj \lambda$ is a second principal surjection then $f=g \circ \sigma$ for some $\sigma \in \fS^{\rm big}$, and so $\alpha_f^{-1}(\fZ)=\alpha_g^{-1}(\fZ)$ since $\fZ$ is $\fS^{\rm big}$-stable. Thus $\cZ$ is well-defined.

Now let $\cZ$ be a subset of $\cX$. We define a subset $\fZ=\Psi(\cZ)$ of $\fX_{\fin}$ by $\fZ=\bigcup_{f \colon 1^{\infty} \psurj \lambda} \alpha_f(\cZ_{\lambda})$. Here the union is taken over all $\infty$-compositions $\lambda$ and all principal surjections $f \colon 1^{\infty} \psurj \lambda$. It is clear that $\fZ$ is $\fS^{\rm big}$-stable.

In the remainder of this section, we see that $\Phi$ and $\Psi$ give bijective correspondence between certain classes of subsets of $\cX$ and $\fX$.

\subsection{The correspondence on C1 subsets}

The following result is the most basic of our correspondences between $\fX$ and $\cX$.

\begin{proposition} \label{prop:C1-corr}
The constructions $\Phi$ and $\Psi$ define mutually inverse bijections
\begin{displaymath}
\xymatrix{
\{ \text{$\fS^{\rm big}$-subsets of $\fX_{\fin}$} \} \ar@<3pt>[r] &
\{ \text{C1 subsets of $\cX$} \} \ar@<3pt>[l] }
\end{displaymath}
\end{proposition}

\begin{proof}
Let $\fZ$ be a $\fS^{\rm big}$-subset of $\fX$ and let $\cZ=\Phi(\fZ)$. We start by verifying that $\cZ$ is indeed a C1 subset. Let $f \colon \lambda \psurj \mu$ be a principal surjection. By definition of $\Phi$, we have $\cZ_{\mu} = \alpha_g^{-1}(\fZ)$ for a principal surjection $g \colon 1^{\infty} \psurj \mu$. Clearly, $g$ factors through $h$, so we can write $g = fh$ for a principal surjection $h \colon 1^{\infty} \psurj \lambda$. Thus $\cZ_{\mu}=\alpha_g^{-1}(\fZ)= \alpha_f^{-1} (\alpha_h^{-1} (\fZ)) = \alpha_f^{-1}(\cZ_{\lambda})$, as required.
	
We now verify that $\fZ=\Psi(\cZ)$.  Consider the map $\pi \colon \coprod_{\lambda, f \colon 1^{\infty} \psurj \lambda } \cX_{\lambda} \to \fX_{\fin}$, where the disjoint union is taken over all isomorphism classes of $\infty$-compositions $\lambda$ and all principal surjections $f \colon 1^{\infty} \psurj \lambda$ (the map $\pi$ on the $(\lambda, f)$-coordinate is given by $\alpha_f$). Putting $Z=\coprod_{\lambda, f \colon 1^{\infty} \psurj \lambda } \cZ_{\lambda}$, we have $Z=\pi^{-1}(\fZ)$ and $\pi(Z)=\Psi(\cZ)$, by definition, and $\pi(\pi^{-1}(\fZ))=\fZ$ since $\pi$ is surjective. Thus $\fZ = \pi(\pi^{-1}(\fZ)) = \pi(Z) = \Psi(\cZ)$.

Now let $\cZ$ be a C1 subset of $\cX$, and put $\cZ'=\Phi(\Psi(\cZ))$. We must show $\cZ=\cZ'$. Again putting $Z=\coprod_{\lambda, f \colon 1^{\infty} \psurj \lambda } \cZ_{\lambda}$ and reasoning as above, we see that $Z \subset \pi^{-1}(\pi (Z)) = \pi^{-1}(\Psi(\cZ))$. Moreover, $\pi^{-1}(\Psi(\cZ))_{(\lambda, f)} = \Phi(\Psi(\cZ))_{\lambda}$, by definition.  Thus $\cZ \subset \cZ'$, so we simply have to establish the reverse inclusion. Let $\lambda$ be an $\infty$-composition and let $z \in \cZ'_{[\lambda]}$ be given. Let $f \colon 1^{\infty} \psurj \lambda$  be a principal surjection. Then $\alpha_f(z)$ belongs to $\Psi(\cZ)$ and has type $\lambda$ (see \S \ref{ss:type}). Thus, by the following lemma, we have $\alpha_f(z)=\alpha_{f'}(z')$ for some $z' \in \cZ_{[\lambda]}$ and some principal surjection $f' \colon 1^{\infty} \psurj \lambda$. Since $z$ and $z'$ have distinct coordinates, we must have $z=\sigma z'$ for some automorphism of $\lambda$. Since $\cZ_{[\lambda]}$ is stable under $\sigma$, we thus have $z \in \cZ_{[\lambda]}$. Thus $\cZ' \subset \cZ$ by Proposition~\ref{prop:C1-contain}.
\end{proof}

\begin{lemma}
Let $\cZ$ be a C1-subset of $\cX$, let $y \in \Psi(\cZ)$, and let $\lambda$ be an $\infty$-weighting. Suppose $y$ has type $\lambda$. Then there exists $z \in \cZ_{[\lambda]}$ and a principal surjection $f \colon 1^{\infty} \psurj \lambda$ such that $y=\alpha_f(z)$.
\end{lemma}

\begin{proof}
By definition of $\Psi$, there is a principal surjection $g \colon 1^{\infty} \psurj \mu$ such that $y = \alpha_g(z')$ for some $z' \in \cZ_{\mu}$. Since $y$ has type $\lambda$, we can find a principal surjection $f \colon 1^{\infty} \psurj \lambda$ such that $y = \alpha_f(z)$ for some $z \in \cX_{[\lambda]}$. Since $\alpha_f(z)=\alpha_g(z')$ and $z$ has distinct coordinates, there is a unique principal surjection $h \colon \mu \psurj \lambda$ such that $z' = \alpha_h(z)$. Since $\cZ$ is C1, we see that $z \in \cZ_{[\lambda]}$, which completes the proof.
\end{proof}

\subsection{The correspondence on C2 subsets}

For a subset $\cZ$ of $\cX$, the subset $\Psi(\cZ)$ of $\fX$ is defined as a union. If $\cZ$ is C2 it can also be described as an intersection:

\begin{proposition} \label{prop:C2-Psi}
Let $\cZ$ be a C2 subset of $\cX$. Then
\begin{displaymath}
\Psi(\cZ) = \fX_{\fin} \cap \bigcap_{f \colon [n] \to [\infty]} \alpha_f^{-1}(\cZ^+_{1^n}),
\end{displaymath}
where the intersection is over all $n \in \bN$ and all injections $f$. In fact, for fixed $N$, it suffices to intersect over $n \ge N$.
\end{proposition}

\begin{proof}
Write $\fZ$ for the right side in the above formula. Let $x \in \Psi(\cZ)$. Then $x=\alpha_g(y)$ for some $\infty$-partition $\lambda$, some $y \in \cZ_{\lambda}$, and some principal surjection $g \colon 1^{\infty} \psurj \lambda$. Let $f \colon [n] \to [\infty]$ be an injection. Then $\alpha_f(x)=\alpha_f(\alpha_g(y))=\alpha_{g \circ f}(y)$. Since $g \circ f$ defines a map of generalized partitions $1^n \to \lambda$, we see that $\alpha_{g \circ f}(y) \in \cZ_{1^n}$. Thus $x \in \fZ$.

Now suppose that $x \in \fZ$. Since $x$ is finitary, we can write $x=\alpha_g(y)$ for some principal surjection $g \colon 1^{\infty} \psurj \lambda$ and some $y \in \cX_{\lambda}$. Now, for any injection $f \colon [n] \to [\infty]$, we have $\alpha_f(x)=\alpha_{g \circ f}(y) \in \cZ_{1^n}$. It is easy to see that any map of generalized partitions $h \colon 1^n \to \lambda$ can be factored as $g \circ f$ for some injection $f \colon [n] \to [\infty]$, and so we see that $\alpha_h(y) \in \cZ_{1^n}$ for all such $h$.

We claim that $y$ is approximable by $\cZ$. Indeed, let $N$ be given; assume $N$ is greater than all the finite parts of $\lambda$. Let $\mu$ be the generalized partition obtained by replacing all infinite parts of $\lambda$ with $N$. Choose a principal surjection $h \colon 1^n \psurj \mu$. Then the function $h$ defines a map of partitions $h \colon 1^n \to \lambda$, and so $\alpha_h(y) \in \cZ_{1^n}$ by the previous paragraph. Since $h$ is a principal surjection to $\mu$ and $\cZ$ is C1, it follows that $y \in \cZ_{\mu}$. This verifies the claim. Since $\cZ$ is C2, it follows that $y \in \cZ_{\lambda}$. Thus $x=\alpha_g(y)$ belongs to $\Psi(Z)$. This completes the proof.
\end{proof}

The following proposition is the next level of our correspondence between $\fX$ and $\cX$.

\begin{proposition} \label{prop:C2-corr}
The constructions $\Phi$ and $\Psi$ define mutually inverse bijections
\begin{displaymath}
\xymatrix{
\{ \text{$\Pi$-closed $\fS$-subsets of $\fX_{\fin}$} \} \ar@<3pt>[r] &
\{ \text{C2 subsets of $\cX$} \}  \ar@<3pt>[l] }
\end{displaymath}
\end{proposition}

\begin{proof}
Suppose $\cZ$ is a C2 subset of $\cX$. Given an injection $f \colon [n] \to [\infty]$, the map $\alpha_f \colon \fX \to \cX_{1^n}$ is continuous for the $\Pi$-topology on $\fX$ and the discrete topology on $\cX_{1^n}$. It follows that $\alpha_f^{-1}(\cZ^+_{1^n})$ is $\Pi$-closed. Thus $\Psi(\cZ)$ is $\Pi$-closed by Proposition~\ref{prop:C2-Psi}.

Conversely, suppose that $\fZ$ is a $\Pi$-closed $\fS$-subset of $\fX_{\fin}$. Then $\fZ$ is $\fS^{\rm big}$-stable (Proposition~\ref{prop:Pi-closure-is-big-stable}), and so $\cZ=\Phi(\fZ)$ is a well-defined C1 subset of $\cX$ by Proposition~\ref{prop:C1-corr}. We now show that $\cZ$ is C2. By Proposition~\ref{prop:C1-corr}, $\cZ$ is C1. Now suppose that $x \in \cX_{\lambda}$ is approximable by $\cZ$, and let $g \colon 1^{\infty} \psurj \lambda$ be a principal surjection. Then for a positive integer $N$, there exists an $\infty$-composition $\mu$ and an injective function $f \colon \langle \lambda \rangle \to \langle \mu \rangle$ such that $\mu_{f(i)} \ge \min(\lambda_i, N)$ for all $i \in \langle \lambda \rangle$ and $y \in \cZ_{\mu}$ such that  $z = \alpha_f(y)$. Thus we may choose a principal projection $h \colon 1^{\infty} \psurj \mu$ such that $\alpha_h(y)$ agrees with $\alpha_g(x)$ on $[N]$. Since $N$ is arbitrary and $\alpha_h(y) \in \fZ$, we see that $\alpha_g(x)$ is in the $\Pi$-closure of $\fZ$. Since $\fZ$ is $\Pi$-closed, $\alpha_g(x) \in \fZ$. It follows that $x \in \cZ$, proving that $\cZ$ is C2. 
\end{proof}

\subsection{The correspondence on C3 subvarieties}

The following theorem is the final level of our correspondence between $\fX$ and $\cX$, and the most important.

\begin{theorem} \label{thm:C3-corr}
Suppose $A$ is noetherian. Then the constructions $\Phi$ and $\Psi$ induce mutually inverse bijections
\begin{displaymath}
\xymatrix{
\{ \text{Zariski closed $\fS$-subsets of $\fX_{\fin}$} \}
\ar@<3pt>[r] &
\{ \text{C3 subvarieties of $\cX$} \} 
\ar@<3pt>[l] }
\end{displaymath}
\end{theorem}

\begin{proof}
Suppose $\fZ$ is a Zariski closed $\fS$-subset of $\fX_{\fin}$, and let $\cZ=\Phi(\fZ)$. Since $\fZ$ is $\Pi$-closed (Proposition~\ref{prop:Zariski-is-coarser}), it follows that $\cZ$ is a C2 subset of $\cX$ (Proposition~\ref{prop:C2-corr}). By definition, $\cZ_{\lambda}=\alpha_f(\fZ)$ where $f \colon 1^{\infty} \psurj \lambda$ is a principal surjection. Since $\alpha_f$ is a map of schemes, it is continuous for the Zariski topology, and so $\cZ_{\lambda}$ is a closed subset of $\cX_{\lambda}$. Thus $\cZ$ is a C3 subvariety of $\cX$.

Conversely, suppose that $\cZ$ is a C3 subvariety of $\cX$. Then $\Psi(\cZ)$ is C2 by Proposition~\ref{prop:C2-Psi}. By Proposition~\ref{prop:C3-plus}, $\cZ^+_{1^n}$ is Zariski closed for $n \gg 0$. Thus by Proposition~\ref{prop:C2-Psi}, we see that $\Psi(\cZ)$ is Zariski closed, as each $\alpha_f^{-1}(\cZ^+_{1^n})$ is Zariski closed. This completes the proof.
\end{proof}

\section{Symmetric subvarieties of $\fX$}

\textit{We assume $A$ is noetherian throughout this section}

\subsection{Classification of symmetric subvarieties of $\fX$}

For a subset $Z$ of $\cX_{\mu}$ define
\begin{displaymath}
\Theta^{\circ}_{\mu}(Z) = \bigcup_{f \colon 1^{\infty} \to \mu} \alpha_f(Z),
\end{displaymath}
where the union is over all maps; this is a subset of $\fX$. It is clear that $\Theta^{\circ}_{\mu}(Z)$ is stable under $\fS^{\rm big}$. We also put $\Theta_{\mu}(Z)=\Theta^{\circ}_{\mu}(Z^e)$.

\begin{proposition}
\label{prop:Theta}
We have $\Theta^{\circ}_{\mu}(Z)=\Psi(\Gamma_{\mu}^{\circ}(Z))$.
\end{proposition}

\begin{proof}
Applying the definitions, we have
\begin{displaymath}
\Theta^{\circ}_{\mu}(Z)=\bigcup_{f \colon 1^{\infty} \to \mu} \alpha_f(Z), \qquad
\Psi(\Gamma_{\mu}^{\circ}(Z)) = \bigcup_{\lambda} \bigcup_{g \colon 1^{\infty} \psurj \lambda} \bigcup_{h \colon \lambda \dashrightarrow \mu} \alpha_g(\alpha_h(Z)).
\end{displaymath}
Since any map $f \colon 1^{\infty} \to \mu$ can be written as a composite $h \circ g$ with $g \colon 1^{\infty} \psurj \lambda$ and $h \colon \lambda \to \mu$, we see that $\Theta^{\circ}_{\mu}(Z) \subset \Psi(\Gamma_{\mu}^{\circ}(Z))$. Conversely, suppose $x \in \Psi(\Gamma_{\mu}^{\circ}(Z))$, and let $\lambda$, $g$, and $h$ be such that $x \in \alpha_g(\alpha_h(Z))$. Let $h=(h_1 \colon \rho \psurj \lambda, h_2 \colon  \rho \to \mu)$. We have $\alpha_h(Z)=\alpha_{h_1}^{-1}(\alpha_{h_2}(Z))$ by definition. Thus we have $x=\alpha_g(y)$ for some $y \in \cX_{\lambda}$ with $\alpha_{h_1}(y) \in \alpha_{h_2}(Z)$, i.e., $\alpha_{h_1}(y)=\alpha_{h_2}(z)$ for some $z \in Z$. Consider the following diagram:
\begin{displaymath}
\xymatrix{
1^{\infty} \ar@{>..>}[rr]^e \ar@{>->}[rd]_g && \rho \ar@{>->}[ld]^{h_1} \ar[rd]^{h_2} \\
& \lambda && \mu }
\end{displaymath}
One easily sees that one can find a principal surjection $e$ making the triangle commute. Let $f=h_2 \circ e$. We have 
\begin{displaymath}
x=\alpha_g(y)=\alpha_e(\alpha_{h_1}(y))=\alpha_e(\alpha_{h_2}(z))=\alpha_f(z)
\end{displaymath}
and so $x \in \Theta^{\circ}_{\mu}(Z)$. This completes the proof.
\end{proof}

It follows from the above proposition that $\Theta_{\mu}(Z)=\Psi(\Gamma_{\mu}(Z))$. We can now classify the $\fS$-irreducible subvarieties of $\fX$:

\begin{theorem} \label{thm:classification-fZ}
We have a bijection
\begin{displaymath}
\left\{ \parbox{18em}{pairs $(\mu,Z)$ with $\mu$ an $\infty$-partition and $Z$ an $\Aut(\mu)$-irreducible closed subset of $\cX_{[\mu]}$} \right\}
\leftrightarrow
\left\{ \parbox{12em}{$\fS$-irreducible Zariski closed bounded subsets of $\fX$} \right\}
\end{displaymath}
given by $(\mu,Z) \mapsto \Theta_{\mu}(Z)$. Moreover, $\Theta_{\mu}(Z)$ is the Zariski closure of
\begin{displaymath}
\bigcup_{f \colon 1^{\infty} \psurj \mu} \alpha_f(Z),
\end{displaymath} where $f$ varies over principal surjections.
\end{theorem}

\begin{proof}
The first statement follows from Theorem~\ref{thm:classification-cZ}, Theorem~\ref{thm:C3-corr} and the previous proposition. For the second statement, let $\fZ$ be the Zariski closure of
\begin{displaymath}
\bigcup_{f \colon 1^{\infty} \psurj \mu} \alpha_f(Z),
\end{displaymath}
where $f$ varies over principal surjections. It is clear that $\Theta_{\mu}(Z)$ contains $\fZ$. By Proposition~\ref{prop:C1-corr}, we see that $\Phi(\fZ)_{\lambda}$ contains $Z$. So by Theorem~\ref{thm:gamma}, we see that $\Phi(\fZ)$ contains $\Gamma_{\lambda}(Z)$. It follows that $\fZ $ contains $\Theta_{\mu}(Z)$ (Theorem~\ref{thm:C3-corr}), completing the proof.
\end{proof}

The inverse to the correspondence in the above theorem is described in the introduction. From the above theorem, we obtain a complete classification of closed $\fS$-subsets of $\fX$:

\begin{corollary}
Let $\fZ$ be a Zariski closed $\fS$-subset of $\fX = \fX_W$. Then there exists a closed subset $W'$ of $W$, finitely many $\infty$ partitions $\mu^1, \mu^2, \ldots, \mu^n$, and $\Aut(\mu^i)$-irreducible closed subvarieties $Z_i$ of $\cX_{[\mu^i]}$ for $ 1 \le i \le n$ such that
\begin{displaymath}
\fZ = \cX_{W'} \cup \bigcup_{1 \le i \le n} \Theta_{\mu^i}(Z_i).
\end{displaymath}
\end{corollary}

\begin{proof}
This follows from the theorem, Proposition~\ref{prop:decomposition-fX}, and Corollary~\ref{cor:fX-components}
\end{proof}

\subsection{The locus $\fX_{\lambda}$}

Let $\lambda$ be an $\infty$-partition. We let $\fX_{[\lambda]}$ be the set of all points of type $\lambda$ (see \S \ref{ss:type}). This locus can also be described as follows:
\begin{displaymath}
\fX_{[\lambda]} = \bigcup_{f \colon 1^{\infty} \psurj \lambda} \alpha_f(\cX_{[\lambda]}),
\end{displaymath}
where the union is taken over all principal surjections $f$. We define $\fX_{\lambda}$ to be the Zariski closure of $\fX_{[\lambda]}$. These sets are perhaps the most fundamental examples of Zariski closed $\fS$-subsets of $\fX$. We now study them in more detail.

\begin{proposition}
Let $\lambda$ be an $\infty$-partition. Then
\begin{displaymath}
\fX_{\lambda}=\Theta_{\lambda}(\cX_{[\lambda]})= \bigcup_{\mu \preceq \lambda} \fX_{[\mu]}.
\end{displaymath}
\end{proposition}

\begin{proof}  Clearly, $\cX_{[\lambda]}^e = \cX_{\lambda}$. So by Proposition~\ref{prop:Theta}, we see that $\Theta_{\lambda}(\cX_{[\lambda]}) = \bigcup_{\mu \preceq \lambda} \fX_{[\mu]}$, where $\fX_{[\mu]} \subset \fX$ denote the set of all points of type $\mu$. By Theorem~\ref{thm:classification-fZ}, it follows that $\Theta_{\lambda}(\cX_{[\lambda]})$ is Zariski closed, and is the Zariski closure of $\fX_{[\lambda]}$. Thus $\Theta_{\lambda}(\cX_{[\lambda]}) = \fX_{\lambda}$. 
\end{proof}

\begin{corollary}
Suppose $W$ is irreducible. Then $\fX_{\lambda}$ is $\fS$-irreducible and corresponds to $(\lambda, \cX_{[\lambda]})$ under the bijection in Theorem~\ref{thm:classification-fZ}
\end{corollary}

We now investigate the defining equations for $\fX_{\lambda}$. Let $\alpha$ be a (finite) partition and let $T$ be a tableau of shape $\alpha$ with distinct entries in $[\infty]$. We define $h_T$ to be the product of $\xi_i-\xi_j$ over those $i,j \in [\infty]$ that appear in $T$ in distinct rows. We write $h_{\alpha}$ for any element $h_T$ with $T$ of shape $\alpha$. Note that if $T$ and $T'$ have shape $\alpha$ then $h_T$ and $h_{T'}$ are in the same $\fS$-orbit, and so $h_{\alpha}$ is well-defined up to the action of $\fS$; in particular, $\Langle h_{\alpha} \Rangle$ is well-defined. If $\alpha \preceq \beta$ then $\Langle h_{\beta} \Rangle \subset \Langle h_{\alpha} \Rangle$. (Recall from \S \ref{ss:noeth} that $\Langle S \Rangle$ denotes the $\fS$-ideal of $R$ generated by $S$.)

We define $I_{\lambda}$ to be the ideal generated by the elements $h_{\alpha}$ with $\alpha \npreceq \lambda$. To be precise, let $S$ be the set of partitions $\alpha$ such that $\alpha \npreceq \lambda$. Then $I_{\lambda}=\Langle h_{\alpha} \Rangle_{\alpha \in S}$. Note that if $\alpha \preceq \beta$ are elements of $S$ then $h_{\beta}$ already belongs to $\Langle h_{\alpha} \Rangle$, and so can be omitted from the list of generators of $I_{\lambda}$. In other words, $I_{\lambda}$ is generated by the $h_{\alpha}$ with $\alpha \in S$ a minimal element. By Proposition~\ref{prop:min-are-finite} and \cite[Proposition~2.1]{catgb},  there are finitely many minimal elements in $S$. 	Suppose $\mu$ is such a minimal element. Let $\ell$ be the number of parts in $\lambda$, and let $e$ be the sum of the finite parts of $\lambda$. Then it is clear that $\mu$ has at most $\ell +1$ parts, and each part is at most $e+1$. In particular, we can find explicit upper bounds on the degree of generation and the number of generators for $I_{\lambda}$.

\begin{theorem} \label{thm:type-loci}
We have $\fX_{\lambda}=V(I_{\lambda})$.
\end{theorem}

We need a few preliminary results.  Let $\lambda$ and $\mu$ be generalized partitions. A {\bf $\lambda$-filling} of $\mu$ is a tableau of shape $\mu$ with entries in positive integers such that the number $i$ appears at most $\lambda_i$ times. We say that a filling is {\bf good} if each number appears in at most one row.

\begin{proposition}
\label{prop:lambda-filling}
$\mu$ admits a good $\lambda$-filling if and only if $\mu \preceq \lambda$.
\end{proposition}

\begin{proof}
First suppose $ \mu \preceq \lambda$. Clearly, $\lambda$ admits a good $\lambda$-filling $T$, namely, the one obtained by filling the $i$th row of the $\lambda$-shaped Young diagram with $i$. By definition, $\mu$ is obtained from $\lambda$ by combining some of the parts and then decreasing some of the parts. By doing the same operations on  $T$, we obtain a good $\lambda$-filling of $\mu$. 
	
Conversely, suppose $T$ is a good $\lambda$-filling of  $\mu$. Suppose $i$ appears $\lambda'_i$ times in  $T$. Then $(\lambda'_1, \lambda'_2, \ldots)$ define a generalized partition (possibly after rearranging to make sure $\lambda'_i \ge \lambda'_{i+1}$). Since $T$ is good, we see that $\mu$ can be obtained by combining some of the parts of $\lambda'$. So $\mu \preceq \lambda'$. Since $T$ is a $\lambda$-filling, we have $\lambda'_i \le \lambda_i$ for all $i$. So $\lambda'$ can be obtained from $\lambda$ by decreasing some of the parts. Thus $\lambda' \preceq \lambda$. We conclude that $\mu \preceq \lambda$, finishing the proof.
\end{proof}

\begin{proposition}
	\label{prop:finite-partition}
	Let $\mu$ and $\lambda$ be generalized partitions. The following are equivalent:
	\begin{enumerate}
		\item $\mu \preceq \lambda$
		\item $\alpha \preceq \mu$ implies $\alpha \preceq \lambda$ for all finite partitions $\alpha$.
	\end{enumerate}
\end{proposition}

\begin{proof}
Clearly, (a) $\Rightarrow$ (b).  Conversely, suppose (b) holds. Pick a number $e$ larger than the sum of the finite parts of $\mu$ and $\lambda$. Let $d_1, d_2$ be the number of infinite parts of $\mu, \lambda$. Let $\alpha$ be the finite partition obtained from $\mu$ by decreasing all infinite parts of $\mu$ to $e$. Clearly, $\alpha \preceq \mu$. By (b), we have $\alpha \preceq \lambda$. By the previous proposition there exists a $\lambda$-good filling $T$ of $\alpha$. Since the size of the $i$th row of $T$ is larger than the sum of finite parts of $\lambda$, for each $i \le d_1$, we can find an element $j_i$ in the $i$th row of $T$ such that $\lambda_{j_i} = \infty$. By goodness of $T$, there is a well-defined injective correspondence $\iota \colon [d_1] \to [d_2]$ given by $i \mapsto j_i$. Now define a $\lambda$ filling $T'$ of the Young diagram of shape $\mu$ as follows:
\begin{itemize}
\item For each $i \le d_1$, fill row $i$ with $\iota(i)$.
\item For each $i > d_1$ fill row $i$ of $T'$ exactly as in $T$.
\end{itemize}
It is clear that $T'$ is a good $\lambda$-filling of $\mu$. Thus by the previous proposition, we have $\mu \preceq \lambda$. This completes the proof of the first assertion.
\end{proof}

\begin{proposition}
We have $V(\Langle h_{\alpha} \Rangle) = \bigcup_{\alpha \npreceq \lambda} \fX_{[\lambda]}$.
\end{proposition}

\begin{proof}
Let $x \in \fX$ be a finitary point of type $\lambda$, and let $\cU$ be the partition of $[\infty]$ induced by $x$, as in \S \ref{ss:type}. Let $U_i$ be the part of $\cU$ of size $\lambda_i$.  We claim that $x \in V(\Langle h_{\alpha} \Rangle)$ if and only if $\alpha$ does not have a good $\lambda$-filling. To see this, first suppose that $x \notin V(\Langle h_{\alpha} \Rangle)$. Then there exists a tableau $T$ of shape $\alpha$, and a $K$-point $y$ above $x$ such that $h_T(y) \neq 0$. Let $T_{i,j}$ be the $j$th entry in row $i$ of $T$, and let $1 \le n_{i,j} \le \ell(\lambda)$ be the unique number such that $T_{i,j} \in U_{n_{i,j}}$. Let $T'$ be the tableau of shape $\alpha$ whose $j$th entry in row $i$ is $n_{i,j}$. Since $h_T(y) \neq 0$, we conclude that $T'$ is a good $\lambda$-filling of $\alpha$. Conversely, suppose that there exists a good $\lambda$-filling $T'$ of $\alpha$. Suppose $i$ appears $\mu_i$ times in $T'$. We know that $\mu_i \le \lambda_i$, and so we can choose a subset $U'_i$ of $U_i$ of size $\mu_i$. For each $i$, we replace instances of $i$ with distinct elements of $U'_i$ in $T'$ to obtain a tableau $T$. Goodness of $T$ shows that $h_T(y) \neq 0$ for any $K$-point $y$ above $x$. This proves the claim.
	
By the claim and Proposition~\ref{prop:lambda-filling}, we conclude that $\fX_{[\lambda]} \subset V(\Langle h_{\alpha} \Rangle) \iff \alpha \npreceq \lambda$. The result now follows immediately.
\end{proof}

\begin{proof}[Proof of Theorem~\ref{thm:type-loci}]  We have already seen that $\fX_{\lambda} = \bigcup_{\mu \preceq \lambda} \fX_{[\mu]}$. Thus it suffices to show that $V(I_{\lambda})  = \bigcup_{\mu \preceq \lambda} \fX_{[\mu]}$. We have
\begin{displaymath}
V(I_{\lambda}) = \bigcap_{\alpha \npreceq \lambda} V(\Langle h_{\alpha} \Rangle) = \bigcap_{\alpha \npreceq \lambda} \bigcup_{\alpha \npreceq \mu} \fX_{[\mu]}.
\end{displaymath}
We thus see that $V(I_{\lambda})=\bigcup_{\mu \in S} \fX_{[\mu]}$ where $S$ is the set of partitions $\mu$ for which $\alpha \npreceq \lambda$ implies $\alpha \npreceq \mu$. Taking the contrapositive, we see that $\mu \in S$ if and only if $\alpha \preceq \mu$ implies $\alpha \preceq \lambda$. By Proposition~\ref{prop:finite-partition}, this condition is  equivalent to $\mu \preceq \lambda$. Thus $S$ consists of those $\mu$ for which $\mu \preceq \lambda$, and so $V(I_{\lambda}) = \bigcup_{\mu \preceq \lambda} \fX_{[\mu]}$.
\end{proof}

\subsection{Example A}

Suppose $\lambda=(\infty, n)$ with $1 \le n < \infty$. Then the set of minimal partitions which are not below $\lambda$ is given by $\{ (1,1,1), (n+1, n+1)  \}$. Thus the ideal $I_{\lambda}$ is generated by the $\fS$-orbits of the polynomials: \begin{align*} h_1 &= (\xi_1-\xi_2)(\xi_2-\xi_3)(\xi_3-\xi_1)\\
	h_2 &= \prod_{0 \le k, l \le n} (\xi_{n+1 -k} - \xi_{2n+2 - l}),
\end{align*} and $V(I_{\lambda})$ consists of points of types $(\infty, a)$ satisfying $0 \le a \le n$. In the case when $n =1$, we see that the degrees of $h_1, h_2$ are strictly larger than 2, but the ideal generated by the orbits of $h_1$ and $(\xi_1-\xi_2)(\xi_3-\xi_4)$ also cuts out points of  types $(\infty)$ and $(\infty, 1)$. Since $I_{\lambda}$ does not contain $(\xi_1-\xi_2)(\xi_3-\xi_4)$ for degree reasons, we conclude that $I_{(\infty, 1)}$ is not radical.


\subsection{Example B}

Suppose $\lambda=(\infty,\infty,2,1)$. Then the set of minimal partitions which are not below $\lambda$ is given by \[ \{ (1,1,1,1,1), (2,2,2,2), (3,3,3,1), (4,4,4) \}.  \] Thus the ideal $I_{\lambda}$ is generated by the $\fS$-orbits of the following polynomials: \begin{align*} h_1 &= \prod_{1 \le i < j \le 5} (\xi_i - \xi_j)\\
	h_2 &= \prod_{1 \le i < j \le 4} \prod_{0 \le k, l \le 1}
(\xi_{2i-k} - \xi_{2j -l})	 \\
	h_3 &= \left(\prod_{1 \le i \le 9 } (\xi_i - \xi_{10})\right) \prod_{1 \le i < j \le 3} \prod_{0 \le k, l \le 2}
	(\xi_{3i-k} - \xi_{3j -l})\\
	h_4 &= \prod_{1 \le i < j \le 3} \prod_{0 \le k, l \le 3}
	(\xi_{4i-k} - \xi_{4j -l}).
\end{align*}
Moreover, $V(I_{\lambda})$ consists of points of types $(\infty, a, b, c)$ satisfying $b+c \le 3$ and $0 \le c \le b \le a$.

\subsection{Equations for general irreducible subvarieties}

We now determine equations for general $\fS$-irreducible subvarieties of $\fX$. Fix an $\infty$-partition $\lambda$ and an $\Aut(\lambda)$-irreducible closed subvariety $Z$ of $\cX_{[\lambda]}$. Let $e$ be the sum of the finite parts of $\lambda$. For a generalized partition $\mu$, we let $\mu^-$ to be the finite partition obtained by reducing all parts larger than $e+1$ to $e+1$, and we let $\mu^s$ to be the unique maximal  partition such that $(\mu^s)^- = \mu^-$. Then $\Lambda_{\le \lambda}^- \coloneq \{ \mu^- \colon  \mu \le \lambda \}$ is a finite set. For each $\mu \in \Lambda_{\le \lambda}^-$, fix a tableau $T_{\mu}$ of shape $\mu$ with distinct entries in $[\infty]$ and a map $\iota_{\mu} \colon A[t_1, \ldots, t_{\ell(\mu^s)}] \to R$ such that for each $i$,  $\iota_{\mu} (t_i) \in \xi_{j}$ for some $j$ in row $i$ of $T_{\mu}$. We set
\begin{displaymath}
I_{\lambda}(Z) = I_{\lambda}  + \sum_{\mu \in \Lambda_{\le \lambda}^-} \Langle h_{T_{\mu}} \iota_{\mu} (I(\cZ_{\mu})) \Rangle
\end{displaymath}
where $\cZ = \Gamma_{\lambda}(Z)^+$. Our main result is the following:

\begin{theorem}
\label{thm:defining-equations}
We have $\Theta_{\lambda}(Z) = V(I_{\lambda}(Z))$.
\end{theorem}

\begin{lemma}
	\label{lem:equations-fin}
Let $\mu \preceq \lambda$ be an $\infty$-partition. The map $\alpha_g \colon \Gamma_{\lambda}(Z)_{\mu}  \to \Gamma_{\lambda}(Z)^+_{\mu^-} $ induced by the natural inclusion $g \colon \mu^- \to \mu$ is an isomorphism. 
\end{lemma}

\begin{proof}
Since $f$ is an isomorphism on the underlying set, we see that $\alpha_g$ is injective. To prove surjectivity, it suffices to treat the case when $\mu = \mu^s$. For that suppose $z \in \Gamma_{\lambda}(Z)^+_{\mu^-}$, then there exists a good correspondence $f = (f_1 \colon \rho \psurj \mu^-, f_2 \colon \rho \to \lambda)$ such that $z \in \alpha_f(Z^e)$ (We note here that Proposition~\ref{prop:finite-union} is valid if $\mu$ is finite). But since $f$ is good and $\mu = \mu^s$, it induces a correspondence $f^s = (f^s_1 \colon \rho^s \psurj \mu, f^s_2 \colon \rho^s \to \lambda)$. Since the underlying maps for $f_i$ is same as that of $f^s_i$ for $i =1,2$, we see that $z \in \alpha_{f^s}(Z^e)$. This shows that $z \in \Gamma_{\lambda}(Z)_{\mu}$, completing the proof. 
\end{proof}

\begin{proof}[Proof of Theorem~\ref{thm:defining-equations}]
First suppose $x \in V(I_{\lambda}(Z))$. Since $I_{\lambda}(Z)$ contains $I_{\lambda}$, we see by Theorem~\ref{thm:type-loci} that $x$ is of type $\mu$ for some $\mu \preceq \lambda$. Let $y$ be a $K$-point above $x$, we claim that $y$ is a $K$-point of $\Theta_{\lambda}(Z)$. Since $y$ is of type $\mu$, by replacing $y$ by $\sigma y$ for some $\sigma \in \fS$, we can assume that $h_{T_{\mu^-}}(y) \neq 0$. Thus we see that $\iota_{\mu^-}(f)(y) = 0$ for each $f \in I(\cZ_{\mu^-})$. By the previous lemma, we have $\cZ_{\mu^-} = \cZ_{\mu}$. Thus $\iota_{\mu^-}(f)(y) = 0$ for each $f \in I(\cZ_{\mu})$. We conclude that $y$ is a $K$-point of $\alpha_{g}(\cZ_{\mu}) $ where $g \colon 1^{\infty} \to \mu$ is any principal surjection such that $g(j) = i$ for each $j$ in row $i$ of $T_{\mu^-}$. Since $\alpha_{g}(\cZ_{\mu}) \subset \Psi(\cZ) = \Theta_{\lambda}(Z) $, it follows that $y$ is a $K$-point of $\Theta_{\lambda}(Z)$. Thus $x \in \Theta_{\lambda}(Z)$.

Conversely, suppose $x \in \Theta_{\lambda}(Z)$, and let $y$ be a $K$-point above $x$. Since $\Theta_{\lambda}(Z)$ is contained in $\cX_{\lambda}$, we see by Theorem~\ref{thm:type-loci} that $f(y) = 0$ for all $y \in I_{\lambda}$. Now suppose $h_{T_{\mu}}(\sigma y)$ is nonzero for some $\mu \in \Lambda_{\le \lambda}^-$ and $\sigma \in \fS$. Then $\sigma y$ is of type $\rho$ such that $\rho^- = \mu$, and there is a principal surjection $g \colon 1^{\infty} \psurj \rho$ such that $g(j) = i$ for each $j$ in row $i$ of $T_{\mu}$ and $y \in \alpha_g(\cZ_{\rho})$. It follows that $\iota_{\mu}(f)(y) = 0$ for each $f \in I(\cZ_{\rho})$. Since $I(\cZ_{\rho}) \supset I(\cZ_{\mu})$, we see that $y$ is a $K$-point of $V(I_{\lambda}(Z))$. Thus $x$ is in $V(I_{\lambda}(Z))$,  completing the proof.
\end{proof}

\subsection{Example C} \label{ss:example-C}

Suppose $\lambda = (\infty, \infty)$, and let $Z \subset \cX_{[\lambda]}$ be the 0-dimensional closed subvariety $\{(0,1), (1,0)\}$. Clearly, $\Aut(\lambda) = \fS_2$ and  $Z$ is $\Aut(\lambda)$-irreducible. We have $e = 0$, and so  $\Lambda_{\le \lambda}^- = \{(1,1), (1)  \}$. It is easy to check that  $\cZ_{(1,1)} = \{(0,1), (1,0), (0,0), (1,1) \}$ and $\cZ_{(1)} = \{(0), (1)   \}$. These are cut out by $\langle t_1(t_1-1),t_2(t_2-1) \rangle$ and $\langle t_1 (t_1 -1) \rangle$ respectively. Fixing $T_{(1,1)}$ with labels 1,2 and $T_{(1)}$ with label 1, we conclude that $I_{\lambda(Z)}$ is generated by the $\fS$-orbits of \begin{align*}
 & (\xi_1 - \xi_2)(\xi_2 - \xi_3)(\xi_3 - \xi_1) \\
 & (\xi_1 - \xi_2)\xi_1(\xi_1 - 1) \\
& (\xi_1 - \xi_2)\xi_2(\xi_2 - 1) \\
& \xi_1(\xi_1 - 1).
\end{align*}
In fact, the first three elements are redundant, as they belong to the $\fS$-ideal generated by the fourth element. These equations implies that $\Theta_{\lambda}(Z)$ consists of points whose coordinates take the values 0 or 1.

\subsection{Containments}

We have parametrized the $\fS$-irreducible closed subsets of $\fX$ in terms of pairs $(\lambda, Z)$. We now show how to determine containment on the $\fX$ side in terms of this parametrization.

\begin{proposition}
	Let $Z_1$ be an $\Aut(\mu)$-irreducible closed subvariety of $\cX_{[\mu]}$, and let $Z_2$ be an $\Aut(\lambda)$-irreducible closed subvariety of $\cX_{[\lambda]}$. Then $\Theta_{\mu}(Z_1) \subset \Theta_{\lambda}(Z_2)$ if and only if
\begin{displaymath}	
Z_1 \subset \bigcup_{\substack{f \colon \mu \dashrightarrow \lambda \\ \text{$f$ is good}}}  \alpha_f(Z_2^e).
\end{displaymath}
\end{proposition}

\begin{proof}
By Theorem~\ref{thm:C3-corr}, the containment $\Theta_{\mu}(Z_1) \subset \Theta_{\lambda}(Z_2)$ is equivalent to the containment $\Gamma^{\circ}_{\mu}(Z_1^e) \subset \Gamma^{\circ}_{\lambda}(Z_2^e)$; by Theorem~\ref{thm:gamma}, this is equivalent to the containment $Z_1^e \subset \Gamma^{\circ}_{\lambda}(Z_2^e)_{\mu}$. By Proposition~\ref{prop:finite-union}, this is equivalent to the containment in the statement of the proposition. 
\end{proof}

\section{The support of a module} \label{s:support}

\textit{We assume $A$ is noetherian throughout this section} \medskip

\subsection{Induction of subvarieties} \label{ss:induction}

Given a subset $\fZ$ of $\fX$, we let $\rI(\fZ)=\bigcup_{\sigma \in \fS} \sigma \fZ$ and let $\ol{\rI}(\fZ)$ be the Zariski closure of $\rI(\fZ)$. The goal of \S \ref{ss:induction} is to show that $\rI(\fZ)$ is not so far from $\ol{\rI}(\fZ)$ in many cases.

To make this precise, we introduce some notation. Suppose $\fZ$ is an $\fS$-irreducible bounded closed $\fS$-subvariety of $\fX$ of type $\lambda$. We define $\fZ^{\top}=\fZ \cap \fX_{[\lambda]}$ to be the subset of $\fZ$ consisting of points of type $\lambda$. It is Zariski dense in $\fZ$. For a general bounded closed $\fS$-subvariety $\fZ$ of $\fX$, write $\fZ=\fZ_1 \cup \cdots \cup \fZ_r$ where the $\fZ_i$ are the $\fS$-irreducible components of $\fZ$, and put $\fZ^{\top}=\fZ_1^{\top} \cup \cdots \cup \fZ_r^{\top}$. Again, this is Zariski dense in $\fZ$.

\begin{proposition}
Let $\cU$ be a partition of $[\infty]$ and let $Z$ be a Zariski closed subset of $\fX_{\cU}$. Then $\ol{\rI}(Z)^{\top} \subset \rI(Z)$.
\end{proposition}

\begin{proof}
It suffices to treat the case where $Z$ is irreducible. Let $\cV$ be the partition of $[\infty]$ defined as follows: $i$ and $j$ belong to the same part if and only if $x_i=x_j$ for all $x \in Z$. Then we have $Z \subset \fX_{\cV} \subset \fX_{\cU}$, and furthermore, $Z \cap \fX_{[\cV]}$ is non-empty. Let $\lambda$ be the partition of $\infty$ associated to $\cV$, and identify $\cX_{\lambda}$ with $\fX_{\cV}$ so that we can regard $Z$ as a closed subvariety of $\cX_{\lambda}$. Let $\cZ=\Gamma_{\lambda}(Z)$ be the C3 subvariety generated by $Z$ and let $\fZ=\Psi(\cZ)$ be the corresponding $\fS$-stable Zariski closed subset of $\fX$. Then, by the nature of the correspondence, we have $\fZ_{[\cV]}=\cZ_{[\lambda]}$, which by Proposition~\ref{prop:gamma-aux3}, is identified with $Z^e \cap \fX_{[\cV]}=\bigcup_{\sigma \in \Aut(\cV)} \sigma(Z \cap \fX_{[\cV]})$. We have
\begin{displaymath}
\fZ^{\top}=\fZ \cap \fX_{[\lambda]}=\bigcup_{\sigma \in \fS} \sigma(\fZ \cap \fX_{[\cV]})
\end{displaymath}
from which it follows that $\fZ^{\top}=\rI(Z \cap \fX_{[\cV]}) \subset \rI(Z)$. Since $\fZ$ is $\fS$-stable and Zariski closed and contains $Z$, we have $\ol{\rI}(Z) \subset \fZ$. On the other hand, the left side contains $\fZ^{\top}$, which is dense in $\fZ$, and so we have equality. We thus conclude that $\ol{\rI}(Z)^{\top} \subset \rI(Z)$.
\end{proof}

Let $\fS_{\ge n}$ be the subgroup of $\fS$ fixing each of $1, \ldots, n-1$. Given an $\fS_{\ge n}$-stable subset $\fZ$ of $\fX$, we think of $\rI(\fZ)$ as a kind of induction of this set to $\fS$. The following is the main result of \S \ref{ss:induction}.

\begin{proposition} \label{prop:ind}
Let $\fZ$ be a bounded $\fS_{\ge n}$-stable Zariski closed subset of $\fX$. Then $\ol{\rI}(\fZ)$ is bounded and $\ol{\rI}(\fZ)^{\top} \subset \rI(\fZ)$.
\end{proposition}

\begin{proof}
It suffices to treat the case where $\fZ$ is irreducible. We can identify $\fS_{\ge n}$ with $\fS$, under which $\fX_W$ is identified with $\fX_{W \times \bA^{n-1}}$. Applying our theory, we see that there is some $\cU$ (a partition of $[\infty]$) such that, putting $Z=\fZ \cap \fX_{\cU}$, we have that $\fZ$ is the Zariski closure of $\bigcup_{\sigma \in \fS_{\ge n}} \sigma Z$.  It follows that $\ol{\rI}(\fZ)=\ol{\rI}(Z)$; as this is contained in $\fX_{\lambda}$, where $\lambda$ is the partition of $\infty$ associated to $\cU$, it is bounded. By the previous proposition, we have $\ol{\rI}(\fZ)^{\top} \subset \rI(Z)$. Since $Z \subset \fZ$, we have $\rI(Z) \subset \rI(\fZ)$, which completes the proof.
\end{proof}

We now give some examples to illustrate some different phenomena that can occur.

\begin{example} \label{ex:induction1}
Let $x$ be the point $(1,0,0,\ldots)$, and let $\fZ=\{x\}$. Then $\fZ$ is a closed $\fS_{\ge 2}$-subvariety of $\fX$. The set $\rI(\fZ)$ is just the $\fS$-orbit of $\fZ$. The Zariski closure of this set contains one additional point, namely the origin; see Example~\ref{example2}. Thus $\rI(\fZ)$ is not Zariski closed, but $\ol{\rI}(\fZ)$ is the $\Pi$-closure of $\rI(\fZ)$.
\end{example}

\begin{example}
Let $\fZ$ be the closed $\fS_{\ge 2}$-subvariety of $\fX$ defined by the equations $\xi_1 \xi_i=1$ for $i \ge 2$. Thus $\fZ$ consists of all points of the form $(a^{-1}, a, a, \ldots)$ with $a$ non-zero. The $\Pi$-closure of $\rI(\fZ)$ contains all points of the form $(a, a, \ldots)$ with $a$ non-zero. We thus see that $\ol{\rI}(\fZ)$ contains the origin. But this does not belong to the $\Pi$-closure of $\rI(\fZ)$, since no point in $\rI(\fZ)$ has zero as a coordinate. Thus $\ol{\rI}(\fZ)$ is strictly larger than the $\Pi$-closure of $\rI(\fZ)$.
\end{example}

\begin{example}
Take $W=\bA^1$, and let $\eta$ be the coordinate on $W$. Let $\fZ$ be the (unbounded) closed $\fS_{\ge 2}$-subvariety of $\fX$ defined by $\eta \xi_1=1$. Then $\fY$ contains all points of the form $(a| a^{-1}, b_2, b_3, \ldots)$, where the bar separates the $\eta$ and $\xi$ coordinates. We thus see that the $\Pi$-closure of $\rI(\fZ)$ consists of all points of the form $(a|b_1,b_2,\ldots)$ with $a$ invertible, which is a dense open subset of $\fX$, and so $\ol{\rI}(\fZ)=\fX$. Seemingly the best we can say about $\rI(\fZ)$ is that it has finite codimension in $\ol{\rI}(\fZ)$.
\end{example}

\subsection{Supports}

We say that an action of $\fS$ on a set $S$ is {\bf smooth} if each $x \in S$ is stabilized by some $\fS_{>n}$. Let $M$ be a smooth $\fS$-equivariant $R$-module. We say that $M$ is {\bf finitely $\fS$-generated} if it is generated, as an $R$-module, by the $\fS$-orbits of finitely many elements. Recall that the {\bf support} of $M$, denoted $\supp(M)$, is the subset of $\fX$ consisting of those prime ideals $\fp$ of $R$ for which the localization $M_{\fp}$ is non-zero. Since $M$ is $\fS$-equivariant, this set is $\fS$-stable. The following is our main theorem on this invariant.

\begin{theorem}
\label{thm:support}
Let $M$ be a smooth $\fS$-equivariant $R$-module that is finitely $\fS$-generated.
\begin{enumerate}
\item There exists a closed $\fS_{\ge n}$-subvariety $\fZ$ of $\fX$, for some $n$, such that $\supp(M)=\rI(\fZ)$ and $V(\ann(M))=\ol{\rI}(\fZ)$.
\item $\supp(M)$ is Zariski dense in $V(\ann(M))$.
\item If $\supp(M)$ is bounded then $V(\ann(M))^{\top} \subset \supp(M)$.
\end{enumerate}
\end{theorem}

\begin{proof}
Let $M$ be generated by the $\fS$-orbits of $x_1, \ldots, x_r$, and let $N$ be the submodule generated by just $x_1, \ldots, x_r$. Let $n$ be such that each $x_i$ is $\fS_{\ge n}$-invariant; thus $N$ is stable under $\fS_{\ge n}$. Let $\fZ=\supp(N)$. Since $N$ is finitely generated as an $R$-module, we have $\fZ=V(\ann(N))$, and so $\fZ$ is a closed $\fS_{\ge n}$-subvariety of $\fX$. Since $M=\sum_{\sigma \in \fS} \sigma N$, we have $\supp(M)=\bigcup_{\sigma \in \fS} \sigma \fZ=\rI(\fZ)$. Since $\supp(M) \subset V(\ann(M))$ and $V(\ann(M))$ is closed, we have $\ol{\rI}(\fZ) \subset V(\ann(M))$.

Write $\ol{\rI}(\fZ)=V(\fa)$ for some radical $\fS$-ideal $\fa$ of $R$. Since $\fZ=V(\ann(N)) \subset V(\fa)$, we have $\fa \subset \rad(\ann(N))$. Write $\fa=\Langle f_1, \ldots, f_s \Rangle$ and let $m$ be such that each $f_j$ is $\fS_{\ge m}$-invariant. For $1 \le i \le r$ and $1 \le j \le s$ and $\sigma \in \fS$, let $k_{i,j}(\sigma)$ be the minimal $k$ such that $(\sigma f_j)^k x_i=0$. Then for fixed $i$ and $j$, the quantity $k_{i,j}(\sigma)$ only depends on the double coset $\fS_{\ge n} \cdot \sigma \cdot \fS_{\ge m}$. Since there are finitely many such double cosets, there is some $k$ such that $k_{i,j}(\sigma) \le k$ for all $i$, $j$, and $\sigma$. We thus have $f_j^k \cdot \sigma x_i=0$ for all $i$, $j$, and $\sigma$, and so $\fb=\Langle f_1^k, \ldots, f_s^k \Rangle$ annihilates $M$. We therefore have $V(\ann(M)) \subset V(\fb)$. Since $V(\fb)=V(\fa)=\ol{\rI}(\fZ)$, we see that $V(\ann(M))=\ol{\rI}(\fZ)$.

We have thus proved (a). Statement (b) follows immediately from (a), while (c) follows from (a) and Proposition~\ref{prop:ind}.
\end{proof}

\begin{example}
We note that $\supp(M)$ can be a proper subset of $V(\ann(M))$. Let $J_i$ be the ideal of $R$ generated by $\xi_i-1$ and $\xi_j$ for $j \ne i$, and consider the module $M=\bigoplus_{i \in [\infty]} Re_i/(J_i e_i)$ with obvious $\fS$-action. Then $\fY=V(J_1)$ is the set considered in Example~\ref{ex:induction1}, and $\supp(M)=\rI(\fY)$ is not Zariski closed; in fact, $\supp(M)$ is not even $\Pi$-closed.
\end{example}

\end{document}